\documentclass[12pt,a4paper]{amsart}

\usepackage{amsmath,amsthm,amssymb,mathtools,amsfonts,mathrsfs}
\usepackage{xspace,xcolor}
\usepackage[breaklinks,colorlinks,citecolor=teal,linkcolor=teal,urlcolor=teal,pagebackref,hyperindex]{hyperref}
\usepackage[alphabetic]{amsrefs}
\usepackage{tikz-cd}
\usepackage[all]{xy}
\usepackage{bbm}
\usepackage{MnSymbol}
\usepackage{stmaryrd}

\usepackage[bbgreekl]{mathbbol}

\DeclareSymbolFontAlphabet{\mathbb}{AMSb}
\DeclareSymbolFontAlphabet{\mathbbl}{bbold}

\newtheorem{theorem}{Theorem}[section]
\newtheorem{prop}[theorem]{Proposition}
\newtheorem{proposition}[theorem]{Proposition}

\newtheorem{thm}[theorem]{Theorem}
\newtheorem{lemma}[theorem]{Lemma}

\theoremstyle{definition} 
\newtheorem{defn}[theorem]{Definition}
\newtheorem{definition}[theorem]{Definition}

\newtheorem{remark}[theorem]{Remark}

\newcommand{\qu}{/\kern-.7ex/}
\newcommand{\lqu}{\backslash \kern-.7ex \backslash}%left
\newcommand{\on}{\operatorname} 
\newcommand{\Aut}{\on{Aut}}

\newcommand{\ev}{\on{ev}}

\newcommand{\NE}{\on{NE}}

\newcommand{\bM}{\overline{\mathcal M}}

\title{Relative mirror symmetry for non-Fano varieties}

\author{Fenglong You}
\address{School of Mathmatical Sciences \\ University of Nottingham, \\University Park \\Nottingham, NG7 2RD \\United Kingdom.}
\email{fenglong.you@nottingham.ac.uk}

\thanks{}


\begin{document}
\date{\today}

\begin{abstract} 
Given a smooth projective variety $X$ with a smooth anticanonical divisor $D$, we study mirror symmetry for the log Calabi--Yau pair $(X,D)$ without assuming that $D$ is nef. We consider the mirror proper Landau--Ginzburg model $(\check X,W)$ from the intrinsic mirror construction of Gross--Siebert. We examine the relationship between the regularized quantum period of $X$ and the classical period of $W$, and identify the discrepancy between them as originating from curve counts in $D$, governed by the mirror map associated with $D$. We also obtain an explicit formula for the proper potential $W$ that encodes this discrepancy. In the end, we show that the quantum period, together with the mirror map, gives exactly the same information as the proper potential.
\end{abstract}

\maketitle 

\tableofcontents

\section{Introduction}

Given a smooth projective variety $X$ with a smooth anticanonical divisor $D$, the mirror of the log Calabi--Yau pair $(X,D)$ is a Landau--Ginzburg (LG) model $(\check X, W)$, where $W: \check X\rightarrow \mathbb C$ is called the potential and is proper. In the intrinsic mirror construction of the Gross--Siebert program \cite{GS19}, the mirror $\check X$ is constructed as the proj of the degree zero part of the relative quantum cohomology of a maximally unipotent monodromy degeneration of $(X,D)$. The proper potential $W$ is also naturally defined in intrinsic mirror symmetry as the theta function $\vartheta_1$ from the degree zero part of the relative quantum cohomology of $(X,D)$. The proper potential was first computed in \cite{GRZ} for toric del Pezzo surfaces. It was generalized to varieties with smooth nef anticanonical divisors in \cite{You22}. 

The Landau--Ginzburg potential can be used to define classical periods. When $X$ is a Fano variety, the quantum period of $X$, defined in \cite{CCGGK}. Mirror symmetry for Fano varieties predicts that the classical periods coincide with the quantum periods. The relation between quantum periods and proper potentials for Fano varieties is also studied in \cite{You22} and \cite{Johnston25}.

In this article, we study mirror symmetry for a log Calabi--Yau pair $(X,D)$ such that $D$ is not necessarily nef. We study both the quantum period and the proper potential for the pair $(X,D)$.

The main examples to keep in mind are blow-ups of Fano varieties. One of the motivations is to study mirror symmetry for the Tyurin degeneration of a Calabi--Yau variety $X$. Under the degeneration, we have two log Calabi--Yau pairs $(X_1,D)$ and $(X_2,D)$. In general, $-K_{X_1}$ and $-K_{X_2}$ are not necessarily nef. These varieties $X_1$ and $X_2$ are called quasi-Fano varieties in \cite{DHT}. 

A well-known example is the Tyurin degeneration of a smooth quintic threefold $Q_5$:
\[
Q_5\leadsto Q_4\cup_{K3}\on{Bl}_C\mathbb P^3,
\] 
where $Q_4$ is a smooth quartic threefold; $\on{Bl}_C\mathbb P^3$ is a blow-up of $\mathbb P^3$ along the intersection $C$ of a quartic and a quintic hypersurface in $\mathbb P^3$; and $K3=Q_4\cap \on{Bl}_C\mathbb P^3$ is a smooth quartic $K3\subset \mathbb P^3$. Mirror symmetry for $(Q_4,K3)$ is well-understood because $Q_4$ is a Fano variety. However, mirror symmetry for $(\on{Bl}_C\mathbb P^3, K3)$ is more intriguing as $\on{Bl}_C\mathbb P^3$ is not Fano.  Such a pair often appears when one studies mirror symmetry for Tyurin degenerations.  
In order to understand mirror symmetry of Tyurin degenerations \cite{DHT}, we need to study relative mirror symmetry for log Calabi--Yau pairs $(X_i,D)$ where $-K_{X_i}$ are not nef. 

%\subsection{Main results}

 When $X$ is Fano, the quantum period is defined as a generating function of one-point genus zero Gromov--Witten invariants of $X$. The general prediction is that the classical periods are equal to the regularized quantum periods. The mirror LG model $(\check X,W)$ depends on the choice of the anticanonical divisor. Let $(X,D)$ be a smooth log Calabi--Yau pair and $(\check X, W)$ be its mirror LG model in intrinsic mirror construction. One can define the classical period using the proper potential $W$. When $-K_X$ is nef, we have the following.

 \begin{thm}[=Theorem \ref{thm-class-quantum}]\label{intro-thm-class-quantum}
Let $X$ be a smooth projective variety and $D$ be a smooth anticanonical divisor of $X$. Suppose $D$ is nef, and the classical period defined by the proper potential $W$ coincides with the regularized quantum period of $X$.
\end{thm}
This refines the result in  \cite{TY20b}*{Theorem 1.14} and \cite{TY20b}*{Remark 6.4}, where the result was up to the relative mirror map because the relative mirror theorem was used to prove it. In Proposition \ref{prop-period-mirror-map}, we show that the effect of the relative mirror map is trivial on the classical period, therefore \cite{TY20b}*{Remark 6.4} is consistent with Theorem \ref{intro-thm-class-quantum}. In Section \ref{sec:quan-period-degen}, we also use the degeneration formula to prove Theorem \ref{intro-thm-class-quantum} directly.

When $-K_X$ is not nef, we can also try to use either the relative mirror theorem or the direct degeneration analysis to compute classical periods. New phenomena appear in both methods. The discrepancy between the regularized quantum periods and the classical periods is the degeneration contributions of curves that are mapped into the divisor $D$. Fortunately, the discrepancy can be encoded in the mirror map of the relative-local model $(Y,D_0)$, where $Y=\mathbb P(\mathcal O_D\oplus N_D)$ and $D_0$ is the zero divisor of $Y$. The pair $(Y,D_0)$ comes from the degeneration of $X$ to the normal cone of $D$. 

Besides the direct degeneration analysis, the results of this article also require careful studies of the relative mirror theorem for log Calabi--Yau pairs \cite{You25}, which was also obtained from the degeneration analysis. In Section \ref{sec:rel-mirror}, we first recall the relative mirror theorem in \cite{You25}.  We focused on the Fano case in \cite{You25}. Here we examine the non-Fano case more carefully. After a thorough examination of the mirror maps, we obtained an explicit mirror formula in Theorem \ref{thm-rel-I-cy}:
\begin{align*}
&J_{(X,D)}(\tau(y),z)\\
= &  \sum_{\beta\in \NE(X)} J_{X,\beta}(\tau_D(y,z),z)y^\beta\frac{\prod_{ a\leq D\cdot\beta}(D+az)}{\prod_{ a\leq 0}(D+az)} \prod_{\text{if }D\cdot\beta>0}\frac{1}{D+D\cdot\beta z}[1]_{-D\cdot\beta},
\end{align*}
where the relative mirror map: $\tau(y)$, and the mirror map in $D$: $\tau_D(y,z)$ are computed in (\ref{mirror-map}) and (\ref{tau-D}) respectively.

We state the relationship between the regularized quantum periods and the classical periods as follows.
\begin{theorem}[=Theorem \ref{thm-classical-quantum}]
Let $X$ be a smooth projective variety and $D$ be a smooth anticanonical divisor of $X$. The classical period defined by the proper potential $W$ coincides with the regularized quantum period of $X$ after imposing the mirror map in $D$. More precisely, we have
\[
\sum_{n\geq 0}[W^n]_{\vartheta_0}=1+\sum_{d\geq 2}\sum_{\substack{\beta\in \NE(X)\\d=D\cdot\beta}} (D\cdot\beta)!\left[\sum_{l\geq 0,\beta^\prime\in \NE(X)}\frac{1}{l!}\langle [\on{pt}]\psi^{d-2},\tau_D,\ldots,\tau_D\rangle_{0,l+1,\beta^\prime}^X q^{\beta^\prime}\right]_{q^\beta} q^\beta, 
\]
where $\tau_D$ is the mirror map for $(Y,D_0)$ and $q^\beta=t^\beta x^{-D\cdot\beta}$.
\end{theorem}

In dimensions $2$ and $3$ we have a more refined result which coincides with the result in \cite{BGL} in dimension $2$.
\begin{theorem}
Let $X$ be a smooth projective variety of complex dimension $3$ or $2$, and let $D$ be a smooth anticanonical divisor of $X$. The classical period defined by the proper potential $W$ coincides with the regularized quantum period of $X$.
\end{theorem}

We also compute the proper potential for smooth log Calabi--Yau pairs. In order to compute the proper potential, we generalize some identities of relative Gromov--Witten invariants with negative contact orders from \cite{You22} to the case where the divisor $D$ is not nef. Some of the identities are similar to those in \cite{You22}, but there are also new invariants that we need to consider because $D$ is not nef. Therefore, the proofs are not just simple generalizations of that of \cite{You22} and we refer the reader to Section \ref{sec:iden-rel-GW} for more details. The relative mirror theorem together with these identities allows us to obtain the following formula for the proper potential.

\begin{theorem}[=Theorem \ref{thm-proper-potential}]
Given a smooth log Calabi--Yau pair $(X,D)$, the proper LG potential is
\begin{align}
W:=x+\sum_{n=1}^{\infty}\sum_{\beta:D\cdot\beta=n+1}n\left\langle [1]_1,[\on{pt}]_{n}\right\rangle_{0,2,\beta}^{(X,D)}t^\beta x^{-n}=x\exp(g(y(q))),
\end{align}
where $y(q)$ is the inverse of the relative mirror map (\ref{iden-rel-mirror-map}) and 
\[
g(y)=\sum_{\substack{\beta\in \NE(X) \\ d=D\cdot\beta\geq 2}}\left[\sum_{l\geq 0,\beta^\prime\in \NE(X)}\frac{1}{l!}\langle [\on{pt}]\psi^{d-2},\tau_D,\ldots,\tau_D\rangle_{0,1+l,\beta^\prime}^X y^{\beta^\prime}\right]_{y^\beta}y^\beta (d-1)!.
\]
\end{theorem}

\begin{remark}
A notable difference of our main results from the case where $D$ is nef is that we also need to impose the mirror map in $D$. This corresponds to contributions of curves that are deformed into $D$. Therefore, the extra degeneration contributions can be considered as providing an enumerative meaning of the inverse of the mirror map in $D$ in the sense that these contributions vanish when applying the mirror map. These extra terms involve relative invariants of $(Y,D_0)$ and we may think about them as certain disk countings in $N_D=Y\setminus D_0$. There have been some results about the enumerative meaning of the mirror maps. See, for example, \cite{CLT}, \cite{CLLT}, \cite{CCLT} and \cite{You20}. But these are with semi-positive assumptions. The mirror maps and the enumerative meanings (these extra degeneration contributions) are much more complicated here because we do not have such assumptions.
\end{remark}

Furthermore, in Section \ref{sec:bell-poly}, we use the Lagrange inversion theorem and identities of Bell polynomials to show that quantum periods together with the mirror map in $D$ give equivalent information as the proper LG potentials. This is different from the Fano case studied in \cite{GRZ}, \cite{You22} and \cite{Johnston25}, where there are no extra mirror maps involved. In the non-Fano case, there is a discrepancy between quantum periods and proper potentials. We show that the discrepancy comes from the contributions of curves that are mapped into $D$ which is controlled by the mirror map in $D$.

\section*{Acknowledgement}
The author would like to thank Per Berglund, Tim Gräfnitz, Michael Lathwood and Siu-Cheong Lau for helpful discussions on related topics.

\section{Preliminary in relative Gromov--Witten theory}

We give a brief review of the definition of genus zero relative Gromov--Witten theory with negative contact orders from \cite{FWY}.

Let $X$ be a smooth projective variety and $D\subset X$ be a smooth divisor. We consider a topological type: 
\[
\Gamma=(0,m,\beta,\vec k),
\]
where $0$ means genus zero, $m$ is the number of markings, $\beta\in \NE(X)$ is the curve class and 
the contact orders of the markings are 
\[
\vec k=(k_1, \ldots, k_m)\in \mathbb Z^{m}
\]
with
\[
\sum_{i=1}^m k_i=D\cdot\beta.
\]
We also set 
\[
m_-:=\{i\in \{1,\ldots,m\}: k_i<0\}; 
\]
\[
m_+:=\{i\in \{1,\ldots,m\}: k_i>0\};
\]
\[
m_\pm=m_-+m_+.
\]

An admissible bipartite graph of topological type $\Gamma$ was defined in \cite{FWY}*{Definition 4.8}. Relative Gromov--Witten invariants in \cite{FWY} are defined as a sum of the contributions from the set $\mathcal B_\Gamma$ of connected admissible bipartite graphs of topological type $\Gamma$.
There is also a relative Gromov--Witten cycle $\mathfrak c_\Gamma(X/D)$ of the pair $(X,D)$ of topological type $\Gamma$ defined in \cite{FWY}*{Definition 5.3}. We refer the precise definition of $\mathfrak c_\Gamma(X/D)$ to \cite{FWY} and just state the following result.
\begin{proposition}[=\cite{FWY}, Proposition 3.4]
\[\mathfrak c_\Gamma(X/D) \in A_{d}(\overline{M}_{0,m}(X,\beta)\times_{X^{m_{\pm}}} D^{m_\pm}),\] where 
\[d=\mathrm{dim}_{\mathbb C}X-3+\int_{\beta} c_1(T_X(-\mathrm{log} D)) + m-m_-. \]
\end{proposition}

The ring of insertions (state space) of relative Gromov--Witten theory is defined as 
\[
\mathfrak H:=\bigoplus_{i\in \mathbb Z} \mathfrak H_i,
\]
where
\[
\mathfrak H_0:= H^*(X) \quad \text{and } \mathfrak H_i:=H^*(D) \text{ for } i\in \mathbb Z\setminus \{0\}.
\]
Let
\begin{align*}
   \alpha_i\in \mathfrak H_0=H^*(X), \text{ if } k_i=0; 
\quad 
     \alpha_i\in \mathfrak H_{k_i}=H^*(D), \text{ if }k_i\neq 0;
\quad
a_i\in \mathbb Z_{\geq 0}.
\end{align*}

\begin{defn}[\cite{FWY}, Definition 5.7]\label{rel-inv-neg}
The genus zero relative Gromov--Witten invariant of topological type $\Gamma$ is
\begin{align}\label{iden-rel-inv-neg}
\langle \prod_{i=1}^m \tau_{a_i}(\alpha_i)  \rangle_{\Gamma}^{(X,D)} =  \displaystyle\int_{\mathfrak c_\Gamma(X/D)} \prod\limits_{i=1}^{m} \bar{\psi}_i^{a_i}\ev_{i}^*\alpha_i.
\end{align}
\end{defn}
The class $\mathfrak c_\Gamma(X/D)$ is quite complicated when there are several markings with negative contact orders. When there is only one marking with negative contact order, it is quite simple. We may assume that $\alpha_1$ is the insertion that corresponds to the unique negative contact marking. Then the relative invariant with one negative contact order can be written as
\begin{align*}
    &\langle \prod_{i=1}^m \tau_{a_i}(\alpha_i)  \rangle_{\Gamma}^{(X,D)} \\
    =& \sum_{\mathfrak G\in \mathcal B_\Gamma} \frac{\prod_{e\in E}d_e}{|\Aut(E)|}\sum \langle  \prod_{i\in S_{1}}\tau_{a_i}(\alpha_i),\tau_{a_1}\alpha_1,\eta\rangle^\sim\langle \check{\eta}, \prod_{i\in S_{2}}\tau_{a_i}(\alpha_i)  \rangle^{\bullet, (X,D)},
\end{align*}
where 
\begin{itemize}
\item $E$ is the set of edges; 
\item $d_e$'s are degrees of the edges; 
\item $\Aut(E)$ is the permutation group of the set $\{d_1,\ldots, d_{|E|}\}$; 
\item a cohomology weighted partition $\eta$ consists of an unordered set of pairs,
\[
\left\{  (c_1,\eta_1),\ldots, (c_l,\eta_l)\right\},
\]
where $\sum_i c_i$ is an unordered partition of $D\cdot\beta$ and $\eta_i\in H^*(D)$. The dual cohomology weighted partition $\eta^\vee$ consists of
\[
\left\{  (c_1,\eta_1^\vee),\ldots, (c_l,\eta_l^\vee)\right\},
\]
where $\eta_i^\vee$ is Poincar\'e dual of $\eta_i$ in $H^*(D)$.
\item the invariant $\langle \cdots\rangle^\sim$ is the rubber invariant of $D$;
\item the first marking becomes a marking with positive contact order $|k_1|$ with the divisor $D_0$ ( where $X$ and $Y=\mathbb P(N_D\oplus \mathcal O_D)$ glue along this divisor under the degeneration to the normal cone). Here, there is only one such rubber for each graph; $S_1\sqcup S_2=\{2,3,\ldots,m\}$.
\item The second sum is over all splittings of $\beta$, all splittings $S_1\sqcup S_2=\{2, 3,\ldots, m\}$, and all cohomology weighted partitions $\eta$.
\end{itemize} 

To keep track of the contact orders, we usually replace $\alpha_i$ with $[\alpha_i]_{k_i}$ on the LHS of (\ref{iden-rel-inv-neg}).

It was proved in \cite{FWY} that the genus zero relative Gromov--Witten theory with negative contact orders satisfies several nice structural properties such as the WDVV equations, topological recursion relation (TRR), relative quantum cohomology ring, Givental formalism, etc.

\section{Relative mirror theorem}\label{sec:rel-mirror}

We recall the relative mirror theorem for a smooth pair $(X,D)$ in \cite{You25}, where we do not assume that $D$ is an anticanonical divisor of $X$.  We have the following:

\begin{theorem}[=\cite{You25}*{Corollary 3.6}]\label{thm-rel-I}
The $I$-function $I_{(X,D),\on{amb}}(y,\tau,z)$ for the smooth pair $(X,D)\subset (P,X_{\infty})$:
\begin{equation*}
\begin{split}
I_{(X,D),\on{amb}}(y,\tau,z)=I_++I_-,
\end{split}
\end{equation*}
where
\[
I_+:=\sum_{\beta\in \NE(X), D\cdot\beta>0} J_{X,\beta}(\tau,z)y^\beta\frac{\prod_{ a\leq D\cdot\beta}(D+az)}{\prod_{ a\leq 0}(D+az)} \frac{1}{D+D\cdot\beta z}[1]_{-D\cdot\beta}\cup h
\]
and 
\[
I_-:=\sum_{\beta\in \NE(X), D\cdot\beta\leq 0} J_{X,\beta}(\tau,z)y^\beta\frac{\prod_{ a\leq D\cdot\beta}(D+az)}{\prod_{ a\leq 0}(D+az)} [1]_{-D\cdot\beta}\cup h,
\]
after applying the mirror map in $D$, lies in  $\iota_*\mathcal L_{(X,D)}$, where $\mathcal L_{(X,D)}$ is Givental's Lagrangian cone of $(X,D)$.
\end{theorem}

In this section, we give a more detailed explanation of Theorem \ref{thm-rel-I}, and then we specialize to the log Calabi--Yau case. There are two types of mirror maps that can appear in Theorem \ref{thm-rel-I}. We now study these mirror maps. 

\subsection{The mirror map in $D$}
Let $Y:=\mathbb P(\mathcal O_{D}\oplus N_{D/ X})$ and $D_{0}$ be the zero divisor of $Y$.  The $I$-function $I_{(Y,D_{0})}(y,y_0,\tau,z)$ of $(Y,D_{0})$ is
\begin{equation}\label{I-func-Y-D-0}
\begin{split}
   e^{h_0\log y_{0}/z}\sum_{\beta\in \on{NE}(D)}\sum_{k\geq c_1(N_{D/X})\cdot\beta} & J_{D, \beta}(\tau,z)y^{\beta}y_{0}^k[1]_{-k+c_1(N_{D/X})\cdot\beta}
\cdot\left(\frac{\prod_{a\leq 0}(h_{0}+az)}{\prod_{a\leq k}(h_{0}+az)}\right)\\
&\cdot\prod_{k-c_1(N_{D/X})\cdot\beta>0}\frac{1}{h_{0}-c_1(N_{D/X})+(k-c_1(N_{D/X})\cdot\beta)z},
\end{split}
\end{equation}
where $h_{0}:=c_1(\mathcal O_{Y}(1))$ and $\tau \in H^2(D)$. 
We recall that $\tau_{D}$ is the mirror map between the $J$-function of $(Y,D_{0})$ and the $I$-function of $(Y,D_{0})$.

If we assume that $D$ is nef, then the mirror map $\tau_D$ is trivial. In general, the mirror map $\tau_{D}$ may contain non-zero coefficients of positive powers of $z$.

We also claim that the mirror map for $(Y,D_0)$ is the same as the mirror map for $N_D$. When $-K_X$ is nef, it is easy to see why the mirror map is the same as the mirror map for $N_D$. By a direct computation, the mirror map is
\[
\tau_{D}:=\sum_{\substack{\beta\in \NE(D):d:=-K_{D}\cdot\beta\geq 2\\-K_{D}\cdot\beta=-c_1(N_{D/X})\cdot\beta}}\left(\langle [\on{pt}]\psi^{d-2}\rangle_{0,1,\beta}^{D}(-1)^{d-1}(d-1)!\right)y^\beta h_{0}.
\]
 Identifying $h_0$ with the divisor class $[D]$, we have the mirror map for $N_D$. Here we do not assume that $D$ is anticanonical. If $D$ is anticanonical, then $D=-K_X$ is nef and $\tau_D$ is trivial.

Now we consider the case where $-K_X$ is not nef. 
The mirror map appears only when $k<0$. In this case, the product
\[
\frac{\prod_{a\leq 0}(h_0+az)}{\prod_{a\leq k}(h_0+az)}=\prod_{k<a\leq 0}(h_0+az)
\]
includes a factor $h_0$. Recall that, if $k>c_1(N_{D/X})\cdot\beta$, the $I$-function takes values in $H^*(D_0)$ and $\iota^*h_0=0$, where $\iota:D\hookrightarrow X$ is the inclusion map. Hence, we need to have $k=c_1(N_{D/X})\cdot\beta$. Therefore, the $z^{a}$-coefficients take values in $h_0\cup H^*(Y)$ for $a\geq 0$. The coefficients are exactly the same as the corresponding coefficients of the $I$-function of $N_D$, where $h_0$ becomes the divisor class $[D]$. That means that the mirror maps are the same. The mirror map is of the form
\begin{align*}
\tau_D(y,z):=\sum_{k \geq 0}\left[I_{(Y,D_{0})}-z\right]_{z^{k}}z^k=\sum_{k\geq 0,\delta \in H^*(D),\beta\in \NE(D)}\tau_{D,\delta,k,\beta}(\delta\cup h_0) z^k y^\beta,
\end{align*}
where $\left[\cdots\right]_{z^k}$ is the coefficient of $z^k$; $\delta\in H^*(D)$; and $\tau_{D,\delta,k,\beta}$ is the $(\delta\cup h_0) z^k y^\beta$-coefficient of the $I$-function (\ref{I-func-Y-D-0}). Furthermore, it satisfies $$\frac{1}{2}\deg(\delta)+k=(-c_1(N_{D/X})+K_{D})\cdot\beta$$ for the coefficient of $y^\beta$. The degree of $z$ is bounded above by $(-c_1(N_{D/X})+K_{D})\cdot\beta$ for the coefficient of $y^\beta$.

The class $\delta\cup h_0$ can be seen as coming from the class $\iota_*\delta\in H^*(X)$ via the degeneration formula, where $\iota:D\hookrightarrow X$ is the inclusion map. This is because in the degeneration formula, there is a lifting choice such that the insertions $\iota_*\delta$ are distributed to the $Y$-side and become $\delta\cup h_0\in H^*(Y)$. For the curve class $\beta\in \NE(D)$, we consider it as a curve class in $X$ by $\iota_*\beta\in \NE(X)$. By abuse of notation, we still denote it by $\beta$. Therefore, the mirror map $\tau_D$ takes values in $H^*(X)$ before the degeneration:
\begin{align}\label{tau-D}
\tau_D(y,z)=\sum_{k\geq 0,\delta \in H^*(D),\beta\in \NE(D)}\tau_{D,\delta,k,\beta}(y)(\iota_*\delta) z^k y^\beta.
\end{align}
 When we say imposing the mirror map in $D$, the precise meaning is to consider the $J$-function of $X$ with the parameter $\tau=\tau_D(y,z)$:
\begin{align}\label{J-func-mirror-map}
J_X(\tau_D(y,z),z)=z+\tau_D(y,z)+\sum_{\substack{\beta\in \NE(X), l\geq 0,\\ \gamma\in H^*(X)}}\langle [\gamma]\psi^a, \tau_D(y,\psi), \ldots,\tau_D(y,\psi) \rangle_{0,l+1,\beta}^X[\gamma^\vee]\frac{y^\beta}{l!}.
\end{align}

The coefficients of $\tau_{D}(y,z)$ can be computed from the $I$-function (\ref{I-func-Y-D-0}). For example, if $\beta\neq 0$, the coefficient of $y^\beta z^{(-c_1(N_{D/X})+K_{D})\cdot\beta}$ is
    \[
 \langle [\on{pt}]\psi^a\rangle_{0,1,\beta}^D  (-c_1(N_{D/X})\cdot\beta-1)!(-1)^{-c_1(N_{D/X})\cdot\beta-1} D\in H^*(X).
    \]
    When $\beta=0$, the coefficient of the $I$-function is just $z+\tau$. Note that in some cases, the invariants $\langle [\on{pt}]\psi^a\rangle_{0,1,\beta}^D$ are simply trivial. For example, when $D$ is Calabi--Yau, the invariants vanish by the virtual dimension constraint. Indeed, the degree of $z$ is bounded by $(-c_1(N_{D/X})+K_{D})\cdot\beta-2$ in this case.

We write
\[
J_X(\tau_D(y,z),z)=\sum_{\beta\in \NE(X)}J_{X,\beta}(\tau_D(y,z),z) y^\beta.
\]
Note that there are also $y^\beta$ in $\tau_D(y,z)$. It should be taken into account when we extract the coefficient of $y^\beta$ from $J_X(\tau_D(y,z),z)$.

For $D\cdot\beta<0$, we claim that $J_{X,\beta}(\tau_D(y,z),z)$ takes value in $D\cup H^*(X)$. This is because of the following. 
\begin{lemma}
For $\gamma\in H^*(X)$, if $D\cdot\beta<0$ and $\gamma^\vee\not\in D\cup H^*(X)$, then 
\[
\left[\sum_{\beta^\prime\in \NE(X), l\geq 0}\langle [\gamma]\psi^a, \tau_D(y,\psi), \ldots,\tau_D(y,\psi) \rangle_{0,l+1,\beta^\prime}^X[\gamma^\vee]\frac{y^{\beta^\prime}}{l!}\right]_{y^\beta}=0.
\]
\end{lemma}
\begin{proof}
We consider the degeneration of $X$ to the normal cone of $D$. Since $\gamma^\vee\not\in D\cup H^*(X)$, we can assume that the distingushed marking (i.e. the first marking) $p_1$ with insertion $\gamma$ away from $D$, hence is distrbuted to the $X$-side. All the other markings are distributed to the $Y$-side as we explained earlier in this section. 

The degeneration graphs are trees since we are in the genus zero case. At the end of the tree there are vertices with one incident edge. If such vertices are on the $Y$-side, then the degeneration contribution of this vertex is of the form
$$\left(\sum_{l\geq 0}\frac{1}{l!}\langle [\eta^\vee_i]_{k_i},\tau_D,\ldots,\tau_D\rangle_{0,l+1,\beta_i}^{(Y,D_0)}\right)$$ 
which forms the $z^{-1}$-coefficient of the $I$-function:
\begin{align*}
\sum_{\beta\in \on{NE}(D), n\geq 0}y^{\beta}y_0^{n}J_{D,\beta}(\tau_{0,2},z)\frac{\prod_{a\leq 0}(h+{D}+az)}{\prod_{a\leq n+D\cdot\beta}(h+{D}+az)}\left(\prod_{\text{if }n>0}\frac{1}{h+z}\right)[1]_{-n}
\end{align*}
that takes value in $[\eta_i]_{-k_i}$ for $k_i>0$. On one hand, if $n+D\cdot\beta<0$, we have a factor of $h+D$ which is zero in $H^*(D_0)$. On the other hand, we can not have $n+D\cdot \beta> 0$ because this will give an extra factor of $z^{-1}$.  Therefore, we must have $n+D\cdot\beta=0$, then the factor 
$$\frac{\prod_{a\leq 0}(h+{D}+az)}{\prod_{a\leq n+D\cdot\beta}(h+{D}+az)}=1.$$
 Moreover, because $n=k_i>0$, we have a factor of $\frac{1}{h+z}$ which contribute to an extra factor of $z^{-1}$. So the $z^{-1}$-coefficient is zero unless $\beta_i$ is multiple of a class of a fiber. Then all the curve classes are distributed to the $X$-side.

We can also compute the invariants $\left(\sum_{l\geq 0}\frac{1}{l!}\langle [\eta^\vee_i]_{k_i},\tau_D,\ldots,\tau_D\rangle_{0,l+1,\beta_i}^{(Y,D_0)}\right)$ by virtual localization of the standard $\mathbb C^\times$-action on $Y$. Then we will consider the corresponding invariants of $N_D$. Note that the mirror map for $(Y,D_0)$ is the same as the mirror map for $N_D$. Using the $I$-function of $N_D$, the same argument implies that the invariants are zero.

If such vertices are on the $X$-side, then the degeneration contribution is $\langle [\eta^\prime]_k\rangle_{0,1,\beta^\prime}^{(X,D)}$. Because $D$ is anticanonical, $\frac{1}{2}\deg(\eta^\prime)=\dim_{\mathbb C}D-1$. Therefore, $\deg\left((\eta^\prime)^\vee\right)=2$.  

Now we may remove these vertices on the $X$-side with one incident edge and consider this half-edge as a leg that records the marking $\eta^\prime$. We denote this new graph as $\mathfrak G^{\prime}$. If there are vertices of $\mathfrak G^{\prime}$ on the $Y$-side that has one incident edge, then we can show that the degeneration contributions are zero by a similar argument. Then we can conclude that all the curve classes are distributed to the $X$-side. There is one vertex on the $X$-side with several edges. This implies that $D\cdot\beta>0$ which contradicts our assumption that $D\cdot\beta<0$. Hence, we conclude the proof.

\end{proof}

Then the relative $I$-function $I_{(X,D)}(y,\tau_D(y,z),z)$ is
\begin{align}\label{rel-I-func-tau}
    \sum_{\beta\in \NE(X)} J_{X,\beta}(\tau_D(y,z),z)y^\beta\frac{\prod_{ a\leq D\cdot\beta}(D+az)}{\prod_{ a\leq 0}(D+az)} \prod_{\text{if }D\cdot\beta>0}\frac{1}{D+(D\cdot\beta) z}[1]_{-D\cdot\beta}.
\end{align}

\begin{remark}
When $D\cdot\beta<0$, we have a factor $D$ in the denominator which is not well-defined. At the same time, the $J$-function $J_{X,\beta}(\tau_D(y,z),z)$ will take value in $D\cup H^*(X)$. The expression (\ref{rel-I-func-tau}) formally means that the class $D$ in the denominator and the numerator will cancel.
\end{remark}

\subsection{The relative mirror map}\label{sec:rel-mirror-map}
Now we compute the relative mirror map when $D$ is a smooth anticanonical divisor of $X$. We need to extract the $z^{\geq 0}$-coefficient of the $I$-function (\ref{rel-I-func-tau}).  

We need to examine possible contributions of the term $\tau_D(y,z)$ in $$J_X(\tau_D(y,z),z)=z+\tau_D(y,z)+O(z^{-1})$$ to the relative mirror map. Recall that $\tau_D(y,z)$ only appears when $K_X\cdot\beta>0$. Furthermore, as $D$ is Calabi--Yau, the $z$-degree of $\tau_D(y,z)$ is bounded above by $(-c_1(N_{D/X})+K_{D})\cdot\beta-2$, hence it is bounded above by $K_X\cdot\beta-2$. As $D$ is anticanonical, the hypergeometric modification is
\[
\frac{\prod_{ a\leq D\cdot\beta}(D+az)}{\prod_{ a\leq 0}(D+az)}=\frac{1}{\prod_{ D\cdot\beta <a\leq 0}(D+az)}=\frac{1}{z^{-D\cdot\beta-1}}\prod_{ D\cdot\beta <a\leq 0}\frac{1}{1+(D/az)},
\]
where $D\cdot\beta=-K_X\cdot\beta$.
Combining these together, we see that $\tau_D(y,z)$ does not contribute to the mirror map because the highest power of $z$ is $z^{-1}$.  Hence, we can write:
\[
I_{(X,D)}=z+I_0 z^0+O(z^{-1}).
\]

The $z^0$-coefficient $I_0$ of the relative $I$-function in Theorem \ref{thm-rel-I}, denoted by $\tau(y)$, is the (non-extended) relative mirror map:
\begin{equation}\label{mirror-map}
\begin{split}
&\sum_{i=1}^{\mathrm r}p_i \log y_i %+\sum_{\substack{\beta\in \NE(X)\setminus \{0\}:D\cdot\beta<0\\ \gamma\in H^4(D)}} \langle \gamma^\vee \rangle_{0,1,\beta}^D  (-c_1(N_{D/X})\cdot\beta-1)!(-1)^{-c_1(N_{D/X})\cdot\beta-1} q^\beta [\gamma]_{-D\cdot\beta}
\\
+&\sum_{\substack{\beta\in \NE(X) \\ d=D\cdot\beta\geq 2}}\left[\sum_{l\geq 0,\beta^\prime\in \NE(X)}\frac{1}{l!}\langle [\on{pt}]\psi^{d-2},\tau_D,\ldots,\tau_D\rangle_{0,1+l,\beta^\prime}^X y^{\beta^\prime}\right]_{y^\beta}y^\beta (d-1)![1]_{-d}.
\end{split}
\end{equation}

We can state Theorem \ref{thm-rel-I} in the  log Calabi--Yau case as follows.
\begin{theorem}\label{thm-rel-I-cy}
Given a smooth log Calabi--Yau pair, the $I$-function $I_{(X,D)}(y,\tau_D(y,z),z)$ in (\ref{rel-I-func-tau})
is equal to the $J$-function $J_{(X,D)}(\tau,z)$ of $(X,D)$ via the relative mirror map $\tau(y)$ in (\ref{mirror-map}):
\begin{align*}
&J_{(X,D)}(\tau(y),z)\\
= & I_{(X,D)}(y,\tau_D(y,z),z)\\
= &  \sum_{\beta\in \NE(X)} J_{X,\beta}(\tau_D(y,z),z)y^\beta\frac{\prod_{ a\leq D\cdot\beta}(D+az)}{\prod_{ a\leq 0}(D+az)} \prod_{\text{if }D\cdot\beta>0}\frac{1}{D+(D\cdot\beta) z}[1]_{-D\cdot\beta}.
\end{align*}
\end{theorem}

\subsection{The extended $I$-function}

We have a similar theorem when we consider the $S$-extended $I$-function. Here we consider the simplest case where $S:=\{1\}$. That means we consider the following $S$-extended $I$-function. 
\begin{defn}\label{def-relative-I-function-extended}
The $S$-extended $I$-function of $(X,D)$, for $S:=\{1\}$, is:
\begin{equation}\label{rel-I-func-extended}
\begin{split}
&I_{(X,D)}^{S}(y,x_1,\tau_D(y,z),z):=\\
&\sum_{\beta\in \on{NE}(X),k\in \mathbb Z_{\geq 0} }J_{X, \beta}(\tau_D(y,z),z)y^{\beta}\frac{ x_1^{k}}{z^{k}k!}\frac{\prod_{ a\leq D\cdot\beta}(D+az)}{\prod_{ a\leq 0}(D+az)}\prod_{\text{if }D\cdot\beta>k}\frac{1}{D+(D\cdot\beta-k) z}[{1}]_{-D\cdot \beta+k}.
\end{split}
\end{equation}

\end{defn}

Now we assume that $D$ is the smooth anticanonical divisor of $X$. The $S$-extended $I$-function is of the form
\[
I^S_{(X,D)}=z+I_0z^0+O(z^{-1}).
\]
 The $z^0$-coefficient $I_0$, now denoted by $\tau(y,x_1)$, is the relative mirror map:
\begin{equation}\label{mirror-map-extended}
\begin{split}
\tau(y,x_1)=&\sum_{i=1}^r p_i\log y_i+x_1[1]_{1}\\
&+\sum_{\substack{\beta\in \NE(X) \\ d=D\cdot\beta\geq 2}}\left[\sum_{l\geq 0,\beta^\prime}\frac{1}{l!}\langle [\on{pt}]\psi^{d-2},\tau_D,\ldots,\tau_D\rangle_{0,1+l,\beta^\prime\in \NE(X)}^X y^{\beta^\prime}\right]_{y^\beta}y^\beta (d-1)![1]_{-d}.
\end{split}
\end{equation}

The relative mirror theorem for the $S$-extended $I$-function is
\begin{theorem}\label{thm-rel-I-cy-extended}
Given a smooth log Calabi--Yau pair, the $I$-function $I^S_{(X,D)}(y,\tau_D(y,z),z)$ in (\ref{rel-I-func-extended})
is equal to the $J$-function $J_{(X,D)}(\tau,z)$ of $(X,D)$ via the relative mirror map $\tau(y)$ in (\ref{mirror-map-extended}):
\begin{align*}
J_{(X,D)}(\tau(y,x_1),z)= & I^S_{(X,D)}(y,x_1,\tau_D(y,z),z).
\end{align*}
\end{theorem}

\subsection{An example}

We consider the degeneration of the quintic threefold $Q_5$ into a union of a quartic threefold $Q_4$ and a blow-up $\on{Bl} {\mathbb P}^3$ of $\mathbb P^3$ along the complete intersection of a quartic and a quintic surfaces:
\[
Q_5\leadsto Q_4\cup_{\on{K3}}\on{Bl} {\mathbb P}^3.
\]

We study the relative mirror symmetry for $(\on{Bl} {\mathbb P}^3,K3)$. As in \cite{DKY}*{Section 4.1}, we can consider $\on{Bl} {\mathbb P}^3$ as a hypersurface in the toric variety $P:=\mathbb P(\mathcal O_{\mathbb P^3}(-1)\oplus \mathcal O_{\mathbb P^3})$. Let $H\in H^2(P)$ be the pullback of the hyperplane class in $H^2(\mathbb P^3)$ and $h=c_1(\mathcal O_P(1))$ be the first Chern class of $\mathcal O_P(1)$. The quasi-Fano variety $\on{Bl} {\mathbb P}^3$ is a hypersurface of $P$ corresponds to the divisor class $4H+h$. There is a toric divisor of the class $h-H$ in $P$. The smooth anticanonical $\on{K3}$ of $\on{Bl} {\mathbb P}^3$ is a complete intersection of $h-H$ and $4H+h$ in $P$.   

The (non-extended) $I$-function of $(\on{Bl} \mathbb P^3, \on{K3})$ is
\begin{equation}\label{I-func-bl-p-3}
    \begin{split}
        zq_1^{H/z}q_0^{h/z}\sum_{d_1,d_0\geq 0}\left(\frac{\prod_{k=1}^{4d_1+d_0}(4H+h+kz)}{\prod_{k=1}^{d_1}(H+kz)^4\prod_{k=1}^{d_0}(h+kz)}\right)\left(\prod_{\text{if } d_0>d_1}\frac{1}{h-H+(d_0-d_1)z}\right)[1]_{-d_0+d_1}q_1^{d_1}q_0^{d_0}.
    \end{split}
\end{equation}

The $I$-function is of the form
\[
I_1z^1+I_0z^0+O(z^{-1}).
\]
By computation, we have
\[
I_1=\sum_{d_0\leq d_1} \frac{(4d_1+d_0)!}{(d_1!)^4d_0!}[1]_{-d_0+d_1}q_1^{d_1}q_0^{d_0},
\]
and
\begin{align*}
I_0=&H\log q_1 +h\log q_0 +\sum_{d_0>d_1}\frac{1}{d_0-d_1} \frac{(4d_1+d_0)!}{(d_1!)^4d_0!}[1]_{-d_0+d_1}q_1^{d_1}q_0^{d_0}\\
+&\sum_{d_0\leq d_1}\frac{(4d_1+d_0)!}{(d_1!)^4d_0!}\left(\sum_{k=1}^{4d_1+d_0}\frac{1}{k}(4H+h)-4\sum_{k=1}^{d_1}\frac {1}{k}H-\sum_{k=1}^{d_0}\frac{1}{k}h\right)[1]_{-d_0+d_1}q_1^{d_1}q_0^{d_0}.
\end{align*}
Note that, if $d_0\neq d_1$, the $I$-function takes value in $K3$, then $H=h$.
The relation between the $I$-function and the $J$-function is the following:
\[
J_{((\on{Bl} \mathbb P^3, \on{K3}))}(I_0/I_1,z)=I_{((\on{Bl} \mathbb P^3, \on{K3}))}(z)/I_1,
\]
where $\frac{1}{I_1}=\frac{1}{1-(-I_1+1)}=\sum_{k\geq 0}(-I_1+1)^k$.

Now we look at the mirror map in $K3\subset \on{Bl} \mathbb P^3$. This is the mirror map of $(Y,D_0)$, where $Y:=\mathbb P(\mathcal O_{K3}\oplus N_{K3/\on{Bl} \mathbb P^3})$. The normal bundle $N_{K3/\on{Bl} \mathbb P^3}=-N_{K3/Q_4}=\mathcal O_{\mathbb P^3}(-1)|_{K3}$.

The $I$-function of $(Y,D_0)$ is
\begin{equation}
    \begin{split}
        zq_1^{H/z}q_0^{h/z}\sum_{d_1,d_0\geq 0}\left(\frac{\prod_{k=1}^{4d_1}(4H+kz)\prod_{k=\infty}^{0}(h-H+kz)}{\prod_{k=1}^{d_1}(H+kz)^4\prod_{k=\infty}^{d_0-d_1}(h-H+kz)}\right)\left(\prod_{\text{if } d_0>0}\frac{1}{h+d_0z}\right)[1]_{-d_0}q_1^{d_1}q_0^{d_0}.
    \end{split}
\end{equation}
Here we use the $I$-function of $D$ instead of the $J$-function of $D$, where the $I$-function of $D$ is
\begin{equation}
    \begin{split}
        zy_1^{H/z}\sum_{d_1\geq 0}\left(\frac{\prod_{k=1}^{4d_1}(4H+kz)}{\prod_{k=1}^{d_1}(H+kz)^4}\right)y_1^{d_1}=I_1z+I_0+O(z^{-1}),
    \end{split}
\end{equation}
which has non-trivial coefficient of $z^1$. The relation between the $I$-function and the $J$-function is
\[
J_D(\tau(y),z)=I_D(y,z)/I_1,
\]
where the mirror map $\tau(y)=I_0/I_1$.
What we discussed in Section \ref{sec:rel-mirror-map} is when $\tau\in H^*(D)$, and no $az$ terms. That is why the $I$-function (\ref{I-func-bl-p-3}) looks slightly different from the $I$-function in Section \ref{sec:rel-mirror-map}.

\section{Quantum periods and relative mirror maps}

Quantum periods are defined for Fano varieties as generating functions of one-point genus zero Gromov--Witten invariants \cite{CCGGK}. Now we consider quantum periods for varieties that are not necessarily Fano and we define it in a similar way.

\begin{definition}\label{def-quan-period}
The quantum period of a smooth projective variety $X$ is
\[
G_X(t)=1+\sum_{\beta\in \NE(X),D\cdot\beta=d\geq 2}\langle [\on{pt}]\psi^{d-2}\rangle_{0,1,\beta}^X t^d,
\]
where we set $q^\beta=t^d$.
\end{definition}

\begin{definition}
The regularized quantum period of a smooth projective variety $X$ is
\[
\widehat G_X(t)=1+\sum_{\beta\in \NE(X),D\cdot\beta=d\geq 2}\langle [\on{pt}]\psi^{d-2}\rangle_{0,1,\beta}^X d! t^d.
\]
\end{definition}

For mirror symmetry, we usually consider varieties whose anticanonical linear systems are non-empty. Here we assume that $X$ has a smooth anticanonical divisor $D$. We then consider the log Calabi--Yau pair $(X,D)$, where $D$ is not necessarily nef. Following \cite{You22}*{Definition 1.1}, the theta functions for the pair $(X,D)$ are generating functions of two-point relative invariants:
\[
\vartheta_p=x^{p}+\sum_{n=1}^{\infty}\sum_{\beta:D\cdot\beta=n+p}n\left\langle [1]_p,[\on{pt}]_{n}\right\rangle_{0,2,\beta}^{(X,D)}t^{n+p}x^{-n}, \quad \text{for } p\geq 0.
\] 
In the intrinsic mirror construction of Gross--Siebert \cite{GS19}, the proper potential is
\[
W:=\vartheta_1.
\]
The classical period of $W$ is
\begin{align*}
\pi_W(t):=&\sum_{l\geq 0}[W^l]_{\vartheta_0}\\
=& 1+\sum_{l\geq 2}\sum_{\beta\in\NE(X),l=D\cdot\beta}\langle [1]_1,\ldots,[1]_1,[\on{pt}]_0\bar{\psi}^{l-2}\rangle_{0,l+1,\beta}^{(X,D)}t^\beta x^l. 
\end{align*}
When $-K_X$ is nef, these invariants can be computed using the relative mirror theorem \cite{FTY}. We see that classical periods equal the regularized quantum periods:
\[
1+\sum_{l\geq 2}\sum_{\beta\in\NE(X),l=D\cdot\beta}\langle [\on{pt}]\psi^{l-2}\rangle_{0,1,\beta}^X y^\beta l!,
\]
where the equality holds up to the relative mirror map. 

Indeed, we have a more refined result: the effect of the relative mirror map on the classical period is trivial.  More precisely,
\begin{proposition}\label{prop-period-mirror-map}
Given a smooth log Calabi--Yau pair $(X,D)$. We have the following identity
\[
\langle [1]_1,\ldots, [1]_1, [1]_{-k_1},\ldots,[1]_{-k_{l^\prime}},[\on{pt}]_0\bar{\psi}^{l-2}\rangle_{0,l+l^\prime+1,\beta}^{(X,D)}=0,
\]
where
\[
l^\prime>0, D\cdot\beta=l-\sum_{i=1}^{l^\prime}k_i>0 \text{ and }k_i\geq 2 \text{ for all } i\in \{1,\ldots,l^\prime\}.
\]
\end{proposition}
\begin{proof}
We use the definition of relative Gromov--Witten invariants with negative contact orders in \cite{FWY}. By definition, every negative contact marking must be in a rubber moduli $\bM^\sim_{\Gamma_i^0}(D)$ labeled by $\Gamma_i^0$. As the last marking has a point constraint, we can assume that the last marking maps to a point that is away from $D$. 

By the definition in \cite{FWY}, each rubber moduli $\bM^\sim_{\Gamma_i^0}(D)$ has at least one negative contact marking distributed to it.
We consider the rubber moduli labeled by $\Gamma_i^0$. Without loss of generality, we can assume that there are $a$ negative contact orders: $[1]_{-k_1},\ldots,[1]_{-k_a}$ and there are $b$ positive contact orders with insertions: $[1]_1$.  The rubber invariant is
\[
\langle \eta_i, [1]_{k_1},\ldots,[1]_{k_a}| \quad | [1]_1,\ldots,[1]_1\rangle_{\Gamma^0_i}^{\sim, \mathfrak c_{\Gamma}}. 
\]
 We can push the virtual cycle of this rubber moduli forward to the moduli space of stable maps to $D$. Then the insertions $[1]_{k_i}$ become the identity class. By definition, when there are $a$ negative contact orders distributed to this rubber moduli, the class $\mathfrak c_\Gamma$ is of degree $a-1$. But at the same time, we will have $a$ identity classes. Applying the string equation $(a-1)$-times, we see that the invariants are zero unless the degree of the curve class in $D$ is zero and there are $a+2$ markings. As there are only fiber curve classes, we at least need $b>a$ in order to make sure the degrees on both sides of the rubber are the same: $D_0\cdot\beta=D_\infty\cdot\beta$. This makes the number of markings larger than what we required. Therefore, the invariants are zero.
\end{proof}

\begin{remark}
This generalizes \cite{You22}*{Proposition 4.8}. Furthermore, we do not need to assume that $D$ is nef. But when $D$ is not nef, this is not the only mirror map that we need to consider, as we discussed in Section \ref{sec:rel-mirror}.
\end{remark}

When $-K_X$ is nef, Proposition \ref{prop-period-mirror-map} and the relative mirror theorem of \cite{FTY} implies that the classical period defined by the proper potential $W$ coincides with the regularized quantum period of $X$. This is Theorem \ref{thm-class-quantum} which we will state and prove again in Section \ref{sec:quan-period-degen} using a degeneration analysis. 

\section{Quantum periods from degeneration}\label{sec:quan-period-degen}

Now we compute the classical period of the proper potential $W:=\vartheta_1$ using the degeneration formula. 

\subsection{The nef anticanonical case}

When $D$ is nef, Proposition \ref{prop-period-mirror-map} and the relative mirror theorem together imply that the classical periods coincide with the regularized quantum periods. We can also prove it using the degeneration formula as follows.
\begin{thm}\label{thm-class-quantum}
Let $X$ be a smooth projective variety and $D$ be a smooth anticanonical divisor of $X$. Suppose $D$ is nef, and the classical period defined by the proper potential $W$ coincides with the regularized quantum period of $X$.
\end{thm}
\begin{proof}
We prove it using the degeneration to the normal cone of $D$ in $X$ and applying the degeneration formula to the one-point invariant 
\begin{align}\label{inv-1-pt}
\langle [\on{pt}]\psi^{D\cdot\beta-2}\rangle_{0,1,\beta}^X. 
\end{align}
The unique marking $p$ has a point constraint, so we can require it to map to a point that is not in $D$. Then under the degeneration formula, the invariant (\ref{inv-1-pt}) becomes
\begin{align*}
\sum_{\mathfrak G} \frac{\prod_{e\in E(\mathfrak G)} d_e}{|\Aut(\mathfrak G)|}\langle  [\on{pt}]\bar{\psi}^{D\cdot\beta-2}| \eta\rangle^{\bullet, (X,D)}\langle \eta^\vee |\rangle^{\bullet, (Y,D_0)},
\end{align*}
where the sum is over all splittings of $\beta$ and all cohomology weighted partitions $\eta$. 

Let $v_0$ be the vertex over $X$-side where the unique marking $p$ is attached to. Let $\overline{M}_{v_0}$ be the moduli space of relative stable maps associated with the vertex $v_0$. The virtual dimension constraint is
\begin{align*}
\dim([\overline{M}_{v_0}]^{\on{vir}})&=(\dim_{\mathbb C}X-3)+(-K_X-D)\cdot\beta+1+|\eta_{v_0}|\\
&=\frac{1}{2}\deg([\on{pt}])+\sum_{i}\frac{1}{2}\deg(\eta^i_{v_0})+D\cdot\beta-2,
\end{align*}
where
\[
\eta_{v_0}=\{(\eta_{v_0}^1,k_1),\ldots,(\eta_{v_0}^{|\eta_{v_0}|},k_{|\eta_{v_0}|})\}
\] 
is the cohomology weighted partition that corresponds to the edges adjacent to the vertex $v_0$. Since $D$ is anticanonical, the virtual dimension constraint simplifies to
\[
|\eta_{v_0}|=\sum_{i}\frac{1}{2}\deg(\eta^i_{v_0})+D\cdot\beta.
\]
Because we assume that $D$ is nef, we must have 
\[
|\eta_{v_0}|\leq D\cdot\beta_{v_0}\leq D\cdot\beta,
\]
where the first equality holds only when $k_i=1$ for all $i\in \{1,\ldots,|\eta_{v_0}|\}$.

In order to satisfy the virtual dimension constraint, we need to have
\[
|\eta_{v_0}|= D\cdot\beta, \quad k_i=1 \text{ for all }i\in \{1,\ldots,|\eta_{v_0}|\}
\]
and 
\[
\eta_{v_0}^i=[1]\in H^*(D) \text{ for all }i\in \{1,\ldots,|\eta_{v_0}|\}.
\]
If $D$ is ample, then $\beta_{v_0}=\beta$. The degeneration formula becomes
\[
\langle [\on{pt}]{\psi}^{D\cdot\beta-2}\rangle_{0,1,\beta}^X =\frac{1}{(D\cdot\beta)!}\langle [1]_1,\ldots,[1]_1,[\on{pt}]_0\bar{\psi}^{D\cdot\beta-2}\rangle_{0,D\cdot\beta+1,\beta}^{(X,D)}.
\]

If $D$ is nef but not necessarily ample, we still have a factor of $\frac{1}{(D\cdot\beta)!}\langle [1]_1,\ldots,[1]_1,[\on{pt}]_0\bar{\psi}^{D\cdot\beta-2}\rangle_{0,D\cdot\beta+1,\beta}^{(X,D)}$. On the $Y$-side we also have the invariants
\[
\langle [\on{pt}]_1 | \rangle_{0,1,\beta_i}^{(Y,D_0)},
\]
where $D_0\cdot\beta_i=1$. The invariants can be computed from the $I$-function
\begin{align*}
\sum_{\beta\in \on{NE}(D), n\geq 0}y^{\beta}y_0^{n}J_{D,\beta}(\tau_{0,2},z)\frac{\prod_{a\leq 0}(h+{D}+az)}{\prod_{a\leq n+D\cdot\beta}(h+{D}+az)}\left(\prod_{\text{if }n>0}\frac{1}{h+z}\right)[1]_{-n}.
\end{align*}
When $D$ is nef, the mirror map is trivial. The invariants are in the $z^{-1}-$coefficient of the $J$-function that takes value in  $[1]_{-1}$. 

We suppose that $\beta_i$ is not a class of a fiber. To extract non-trivial $z^{-1}$-coefficient of the $I$-function we need $n+D\cdot\beta\leq 0$. Otherwise, the factor 
\[
\frac{\prod_{a\leq 0}(h+{D}+az)}{\prod_{a\leq n+D\cdot\beta}(h+{D}+az)}=\frac{1}{\prod_{0<a\leq n+D\cdot\beta}(h+{D}+az)}.
\]
And we have
\[
\frac{1}{h+D+az}=\frac{1}{az}\frac{1}{1+(h+D)/(az)}=\frac{1}{az}\sum_{k\geq 0}\left(\frac{h+D}{az}\right)^k.
\]
This gives extra negative powers of $z$ and does not contribute to $z^{-1}$-coefficient. On the other hand, $n+D\cdot\beta\leq 0$ is not possible because $D$ is nef and we want $n=1$ here. Therefore, the corresponding coefficient of the $I$-function is zero unless $\beta_i$ is a class of a fiber. Hence these invariants are zero  unless $\beta_i$ is a class of a fiber (in that case the invariants are equal to $1$). 

Note that there are no other vertices over the $X$-side because $D$ is nef and the sum of the degree of the edges adjacent to $v_0$ is already $D\cdot\beta$. Therefore, there are no other degeneration contributions. Again, the degeneration formula becomes
\[
\langle [\on{pt}]{\psi}^{D\cdot\beta-2}\rangle_{0,1,\beta}^X =\frac{1}{(D\cdot\beta)!}\langle [1]_1,\ldots,[1]_1,[\on{pt}]_0\bar{\psi}^{D\cdot\beta-2}\rangle_{0,D\cdot\beta+1,\beta}^{(X,D)}.
\]

\end{proof}

\begin{remark}
The degeneration argument in the proof of Theorem \ref{thm-class-quantum} is consistent with the computation using the relative mirror theorem when $-K_X$ is nef. We will see that it is more complicated when $-K_X$ is not nef.
\end{remark}

\subsection{The non-Fano case}

Now we consider the case where $D$ is not nef. Then the proof in Theorem \ref{thm-class-quantum} does not work. In the degeneration formula, there can be extra terms. For example, there are rational tails:
\[
\langle [\eta_i]_{k_i} | \rangle_{0,1,\beta_i}^{(Y,D_0)}.
\] 
In general, there are also other vertices on the $X$-side and there are several relative markings.

In the proof of Theorem \ref{thm-class-quantum}, these terms are trivial when $D$ is nef. When $D$ is not nef, they are not necessarily trivial. These terms exactly contribute to the mirror maps in $D$ that appear in the relative mirror theorem that we discussed in Section \ref{sec:rel-mirror}.

\begin{theorem}\label{thm-classical-quantum}
Let $X$ be a smooth projective variety and $D$ be a smooth anticanonical divisor of $X$. The classical period defined by the proper potential $W$ coincides with the regularized quatum period of $X$ after imposing the mirror map in $D$. More precisely, we have
\begin{align}\label{iden-classical-quantum}
\pi_W(t):=&\sum_{n\geq 0}[W^n]_{\vartheta_0}\\
=&1+\sum_{d\geq 2}\sum_{\beta\in\NE(X),D\cdot\beta=d} (D\cdot\beta)!\left[\sum_{l\geq 0,\beta^\prime\in \NE(X)}\frac{1}{l!}\langle [\on{pt}]\psi^{d-2},\tau_D,\ldots,\tau_D\rangle_{0,l+1,\beta^\prime}^X q^{\beta^\prime}\right]_{q^\beta} q^\beta, 
\end{align}
where $\tau_D$ is the mirror map in $D$ and $q^\beta=t^\beta x^{-D\cdot\beta}$.
\end{theorem}

\begin{proof}
Recall that the classical period equals
\begin{align*}
&\sum_{n\geq 0}[W^n]_{\vartheta_0}\\
=& 1+\sum_{n\geq 2}\sum_{\beta\in\NE(X): D\cdot\beta=n}\langle [1]_1,\ldots,[1]_1,[\on{pt}]_0\bar{\psi}^{n-2}\rangle_{0,n+1,\beta}^{(X,D)}t^\beta x^{-n}. 
\end{align*}

Similar to the proof of Theorem \ref{thm-class-quantum}, we consider the degeneration formula for the absolute invariants of $X$. The first marking has a point insertion, we can require the first marking $p_1$ is mapped to a point that is away from $D$. So we can require that the first marking $p_1$ is distributed to the $X$-side. Other markings will be distributed to the $Y$-side because the mirror map $\tau_D$ is of the form
\[
\tau_D=\sum_{k\geq 0}\tau_{D,\delta,k}(\iota_*\delta) z^k,
\]
for $\delta\in H^*(D)$ as we explained in (\ref{tau-D}).

Let $\mathfrak G$ be a bipartite graph coming from the degeneration formula. In genus zero, the graph is a tree. There is a vertex $v_0$ on the $X$-side that carries the marking $p_1$. The vertex $v_0$ is connected with the rest of the graph by edges $e_1,\ldots,e_{|E(v_0)|}$, where $E(v_0)$ is the set of edges adjacent to the vertex $v_0$ and $|E(v_0)|$ is the number of elements in $E(v_0)$.

Let $\mathfrak G_i^\prime$ be the subgraph of $\mathfrak G$ that connects with $v_0$ via the edge $e_i$. We write $C_{\mathfrak G_i^\prime}$ the degeneration contribution from $\mathfrak G_i^\prime$. The degeneration formula can be written as
\begin{align*}
&\sum_{l,\beta}\frac{1}{l!}\langle [\on{pt}]\psi^k,\tau_D,\ldots,\tau_D\rangle_{0,l+1,\beta}^X\\
=&\sum_{\mathfrak G}\langle [\on{pt}]_0\psi^k | C_{\mathfrak G_1^\prime}\eta_1,\ldots, C_{\mathfrak G_{|E(v_0)|}^\prime}\eta_{|E(v_0)|}  \rangle_{0,1+|E(v_0)|,\beta}^{(X,D)}, 
\end{align*}
where $\eta_i\in H^*(D)$.

In the degeneration formula we have the principal contribution when there is only one vertex on the $X$-side. This is
\begin{align*}
\sum_{\mathfrak G}& \frac{\prod_{e\in E(\mathfrak G)} d_e}{|\Aut(\mathfrak G)|}\langle [\on{pt}]_0\psi^{D\cdot\beta-2}, [\eta_1]_{k_1},\ldots,[\eta_l]_{k_|E|}\rangle_{0,1+|E|,\beta}^{(X,D)}\\
& \cdot\prod_{i=1}^{|E|}\left(\sum_{l\geq 0}\frac{1}{l!}\langle [\eta^\vee_i]_{k_i},\tau_D,\ldots,\tau_D\rangle_{0,l+1,\beta_i}^{(Y,D_0)}\right).
\end{align*}
The invariants $\left(\sum_{l\geq 0}\frac{1}{l!}\langle [\eta^\vee_i]_{k_i},\tau_D,\ldots,\tau_D\rangle_{0,l+1,\beta_i}^{(Y,D_0)}\right)$ form the $z^{-1}$-coefficient of the $I$-function
\begin{align*}
\sum_{\beta\in \on{NE}(D), n\geq 0}y^{\beta}y_0^{n}J_{D,\beta}(\tau_{0,2},z)\frac{\prod_{a\leq 0}(h+{D}+az)}{\prod_{a\leq n+D\cdot\beta}(h+{D}+az)}\left(\prod_{\text{if }n>0}\frac{1}{h+z}\right)[1]_{-n}
\end{align*}
that takes value in $[\eta_i]_{-k_i}$ for $k_i>0$. Note that we can not have $n+D\cdot \beta> 0$ because this will give an extra factor of $z^{-1}$. If $n+D\cdot\beta<0$, we have a factor of $h+D$ which is zero in $H^*(D_0)$. Therefore, we must have $n+D\cdot\beta=0$, then the factor $$\frac{\prod_{a\leq 0}(h+{D}+az)}{\prod_{a\leq n+D\cdot\beta}(h+{D}+az)}=1.$$ On the other hand, because $n>0$, we have a factor of $\frac{1}{h+z}$ which contribute to an extra factor of $z^{-1}$. So the $z^{-1}$-coefficient is zero unless $\beta_i$ is multiple of a class of a fiber. Then all the curve classes are distributed to the $X$-side, we have $\eta_i=[1]_1$ for all $i$ and $|E|=D\cdot\beta$ by the virtual dimension constraint. The principal contribution is
\[
\frac{1}{(D\cdot\beta)!}\langle [1]_1,\ldots,[1]_1,[\on{pt}]_0\bar{\psi}^{D\cdot\beta-2}\rangle_{0,D\cdot\beta+1,\beta}^{(X,D)}.
\]

Now we would like to show that other contributions are also zero because of the mirror map $\tau_D$. We consider the degeneration contribution $C_{\mathfrak G_i^\prime}$. Without loss of generality, we can consider $C_{\mathfrak G_1^\prime}$. 
The graph  $\mathfrak G_1^\prime$  is a tree because we consider genus zero invariants. At the end of the tree there are vertices with one incident edge. If such vertices are on the $Y$-side, then the degeneration contribution of this vertex is $$\left(\sum_{l\geq 0}\frac{1}{l!}\langle [\eta^\vee_i]_{k_i},\tau_D,\ldots,\tau_D\rangle_{0,l+1,\beta_i}^{(Y,D_0)}\right)$$ that forms the $[\eta]z^{-1}$-coefficient of the $I$-function which is zero as we just computed. If such vertices are on the $X$-side, then the degeneration contribution is $\langle [\eta^\prime]_k\rangle_{0,1,\beta^\prime}^{(X,D)}$. Because $D$ is anticanonical, $\frac{1}{2}\deg(\eta^\prime)=\dim_{\mathbb C}D-1$. Therefore, $\deg\left((\eta^\prime)^\vee\right)=2$.  

Now we may remove these vertices on the $X$-side with one incident edge and consider this half-edge as a leg that records the marking $\eta^\prime$. We denote this new graph as $\mathfrak G_1^{\prime\prime}$. If there are vertices of $\mathfrak G_1^{\prime\prime}$ on the $Y$-side that has one incident edge, then we can show that the degeneration contributions are zero by a similar argument. Then we can conclude that all these contributions $C_{\mathfrak G_i^\prime}$ are zero. 

Therefore, the only non-zero degeneration contribution is
\[
\frac{1}{(D\cdot\beta)!}\langle [1]_1,\ldots,[1]_1,[\on{pt}]_0\bar{\psi}^{D\cdot\beta-2}\rangle_{0,D\cdot\beta+1,\beta}^{(X,D)}.
\]
Hence we conclude the proof.

\end{proof}

\begin{remark}
The extra terms in the degeneration formula of one-point invariants $\langle [\on{pt}]_0\psi^{D\cdot\beta-2}\rangle_{0,\beta,1}^{X}$ are inverse of the mirror map in $D$ in the sense that these terms vanish when we apply the mirror map.
\end{remark}

We can also see it by computing the classical period from Theorem \ref{thm-rel-I} when $D$ is not nef. In this case, we have the mirror map in $D$ in the relative $I$-function. This matches with the degeneration argument above.

In \cite{BGL}, Berglund--Grafnitz--Lathwood studied mirror symmetry for non-Fano varieties in dimension $2$. In this case, we can show that the classical periods still coincide with the regularized quantum periods. Indeed, in dimensions $3$ and $2$, we have the following result.
\begin{theorem}
Let $X$ be a smooth projective variety of complex dimension $3$ or $2$, and let $D$ be a smooth anticanonical divisor of $X$. The classical period defined by the proper potential $W$ coincides with the regularized quantum period of $X$.
\end{theorem}
\begin{proof}
We just need to show that the mirror map in $D$ is trivial. As $D$ is a smooth anticanonical divisor of a threefold or a surface, we know that $D$ is a $K3$ surface or an elliptic curve. The genus zero Gromov--Witten invariants of $D$ are zero. 

The mirror map in $D$ is computed from the $I$-function of $(Y,D_0)$
\begin{align*}
\sum_{\beta\in \on{NE}(D), n\geq 0}y^{\beta}y_0^{n}J_{D,\beta}(\tau_{0,2},z)\frac{\prod_{a\leq 0}(h+{D}+az)}{\prod_{a\leq n+D\cdot\beta}(h+{D}+az)}\left(\prod_{\text{if }n>0}\frac{1}{h+z}\right)[1]_{-n}.
\end{align*}
As the genus zero invariants of $D$ are zero when $D$ is an elliptic curve or a $K3$ surface. 
 Therefore, the mirror map in $D$ is trivial. Hence, the classical period of the proper potential $W$ coincides with the regularized quantum period of $X$ when $X$ is of complex dimension $2$ and $3$.
\end{proof}

\section{The proper potential for non-Fano varieties}

Now we compute the proper potential for non-Fano varieties:

\[
W=\vartheta_1=x+\sum_{n=1}^{\infty}\sum_{\beta:D\cdot\beta=n+1}n\left\langle [1]_1,[\on{pt}]_{n}\right\rangle_{0,2,\beta}^{(X,D)}t^{n+1}x^{-n}.
\] 
Alternatively, it is the same as defining theta functions using three point invariants \cite{You24}:
\[
W=\vartheta_1=x+\sum_{n=1}^{\infty}\sum_{\beta:D\cdot\beta=n+1}\left\langle [1]_1,[1]_{\mathbf b},[\on{pt}]_{-\mathbf b+n}\right\rangle_{0,3,\beta}^{(X,D)}t^{n+1}x^{-n},
\] 
where the last two markings are mid-age markings. Similarly to the strategy in \cite{You22}, we use the relative mirror theorem to compute this function. As $D$ is not nef now, we need to use the relative mirror theorem in \cite{You25}. We recall that the relative $I$-function is
\begin{equation}\label{I-func-X-D-smooth}
\begin{split}
\sum_{\beta\in \on{NE}(X)} J_{X, \beta}(\tau_D(y,z),z) y^{\beta}\frac{\prod_{ a\leq D\cdot\beta}(D+az)}{\prod_{ a\leq 0}(D+az)}\cdot\prod_{\text{if } D\cdot\beta>0}\frac{1}{D+D\cdot\beta z}[1]_{ -D\cdot\beta}.
\end{split}
\end{equation}

The relative mirror map is given by the $z^{0}$-coefficient of the relative $I$-function. We computed in (\ref{mirror-map}) that the relative mirror map is
\begin{equation*}
\begin{split}
&\sum_{i=1}^{\mathrm r}p_i \log y_i\\
+&\sum_{\substack{\beta\in \NE(X) \\ d=D\cdot\beta\geq 2}}\left[\sum_{l\geq 0,\beta^\prime\in \NE(X)}\frac{1}{l!}\langle [\on{pt}]\psi^{d-2},\tau_D,\ldots,\tau_D\rangle_{0,1+l,\beta^\prime}^X y^{\beta^\prime}\right]_{y^\beta}y^\beta (d-1)![1]_{-d}.
\end{split}
\end{equation*}

\subsection{Some identities of relative Gromov--Witten invariants}\label{sec:iden-rel-GW}
To compute the theta function, we need to generalize some identities of relative Gromov--Witten invariants from \cite{You22} to the non-Fano case. We also do not require $D$ to be anticanonical in this section. 

We first prove some lemmas.

\begin{lemma}\label{lemma-several-neg-0}
Let $(X,D)$ be a smooth pair. Given a curve class $\beta$, let $k_i\in \mathbb Z_{>0}$ for $i\in \{1,\ldots, l\}$ such that
\[
D\cdot \beta=k_{l+1}-\sum_{i=1}^l k_i, \quad \text{and } k_{l+1}\geq 0.
\]
Then we have the following relation:
\begin{align}\label{identity-several-neg-0}
\langle  [1]_{-k_1},\cdots, [1]_{-k_l},  [\gamma]_{k_{l+1}}\rangle_{0,l+1,\beta}^{(X,D)}=
\left\{
\begin{array}{cc}
\langle [D]_0,\cdots, [D]_0,[\gamma]_{D\cdot \beta} \rangle_{0,l+1,\beta}^{(X,D)} & \text{if }D\cdot\beta \geq 0\\ 
0 & \text{if } D\cdot\beta< 0, 
\end{array}
\right.
\end{align}
where $\gamma \in H^*(D)$ when $k_{l+1}>0$; $\gamma\in H^*(X)$ when $k_{l+1}=0$ and the invariant is zero.
\end{lemma}

\begin{proof}
We proved it in \cite{You22}*{Proposition 4.6} under the assumption that $D$ is nef. Now we do not assume that $D$ is nef.

Every marking with negative contact order must be distributed to a rubber moduli.  Suppose the last marking $p_{l+1}$ is not distributed to the rubber, then we have the following degeneration contributions:
\begin{align*}
\sum_{\frak G\in \mathcal B_\Gamma}\frac{\prod_{e\in E}d_e}{|\Aut(\frak G)|}\sum \langle  |\,|[1]_{k_1},\ldots [1]_{k_l},\eta\rangle^{\bullet,\sim} \langle  \check{\eta}, [\gamma]_{k_{l+1}} \rangle^{\bullet, (X,D)},
\end{align*}
where the sum is over all splitting of $\beta$ and all cohomology weighted partitions $\eta$.

 For the rubber integral
\[
 \langle  |\,|[1]_{k_1},\ldots [1]_{k_l},\eta\rangle^{\bullet,\sim}, 
\]
we push the rubber moduli forward to the moduli space of stable maps to $D$.
By \cite{JPPZ18}*{Theorem 2}, we have
\[
p_*[\bM^\sim(Y,D_0\cup D_\infty)]^{\on{vir}}=[\bM_{0,m}(D,\pi_*(\beta_1))]^{\on{vir}}.
\]
 The markings $[1]_{k_i}$ become the identity classes $[1]\in H^*(D)$. Then there are $l$ identity classes and the degree of the class $\mathfrak c_\Gamma(X/D)$ is less than or equal to $(l-1)$. When we apply the string equation $(l-1)$-times, the invariant vanishes unless $\pi_*(\beta)=0$.  This implies $D_0\cdot\beta_1=D_\infty\cdot\beta_1$. The absence of relative markings at $D_0$ means that $D_0\cdot\beta_1=0$. This is a contradiction because the intersection number $D_\infty\cdot\beta_1>0$. Therefore, we conclude that the last marking $p_{l+1}$ is also distributed to the rubber moduli. 

Hence the degeneration formula is of the following form:
\begin{align*}
\sum_{\frak G\in \mathcal B_\Gamma}\frac{\prod_{e\in E}d_e}{|\Aut(\frak G)|}\sum \langle [\gamma]_{k_{l+1}} |\,|[1]_{k_1},\ldots [1]_{k_l},\eta\rangle^{\bullet,\sim} \langle  \check{\eta} \rangle^{\bullet, (X,D)},
\end{align*}
where $\eta$ can be empty.
By the same degeneration analysis, we know that the curve class of the rubber moduli should be a multiple of the class of a fiber.

Furthermore, if $\eta$ is not empty, then we need to have $k_{l+1}-\sum_{i=1}^l k_i>0$, otherwise the rubber invariants, hence the invariants on the LHS of (\ref{identity-several-neg-0}), are zero.  The case where $\eta$ is empty is exactly when $k_{l+1}-\sum_{i=1}^l k_i=0$. In this case,  we also have $\pi_*(\beta)=0$ otherwise the invariant is zero. Therefore, we have $\beta=0$ and the LHS of (\ref{identity-several-neg-0}) is equal to $\int_D D\cup \gamma$ by a direct calculation. This is also equal to the RHS (\ref{identity-several-neg-0}).

The remaining proof when $D\cdot\beta>0$ is identical to the proof of \cite{You22}*{Proposition 4.6, the base case I}.

\end{proof}

\begin{remark}
The divisor classes on the RHS of (\ref{identity-several-neg-0}) can be removed by the divisor equation. Likewise, we can add more divisor classes to both sides and the equality still holds.
\end{remark}

\begin{lemma}\label{lemma-several-neg-0-}
Let $(X,D)$ be a smooth pair. Given a curve class $\beta$, let $k_i\in \mathbb Z_{>0}$ for $i\in \{1,\ldots, l\}$ such that
\[
D\cdot \beta=k_{l+1}-\sum_{i=1}^l k_i<0, \text{ and } k_{l+1}<0.
\]
Then we have the following relation:
\begin{align}\label{identity-several-neg-0-}
\langle  [1]_{-k_1},\cdots, [1]_{-k_l},  [\gamma]_{k_{l+1}}\rangle_{0,l+1,\beta}^{(X,D)}=
\langle [D]_0,\cdots, [D]_0,[\gamma]_{D\cdot \beta} \rangle_{0,l+1,\beta}^{(X,D)},
\end{align}
where $\gamma \in H^*(D)$.
\end{lemma}
\begin{proof}
When $l=1$, the LHS of (\ref{identity-several-neg-0-}) is a relative invariant with two negative contact orders. The class $\mathfrak c_{\Gamma}$ is of real degree $2$ and it was computed in \cite{FWY}*{Example 5.6}. As in the proof of Lemma \ref{lemma-several-neg-0}, two markings with negative contact orders should be in the same rubber. Therefore, for the invariant that we consider, there is only one type of bipartite graph instead of three types in \cite{FWY}*{Example 5.6}. There are $\bar{\psi}$-classes $d_e\bar{\psi}_e$ in $\mathfrak c_{\Gamma}$ that will cancel with the identity class in the first marking. We end up with the multiplicity $d=D\cdot\beta$. Hence, it agrees with the RHS of (\ref{identity-several-neg-0-}) by the divisor equation. 

When $l>1$, all the markings are also distributed to the same rubber. The class $\mathfrak c_{\Gamma}$ is also explicitly described in \cite{FWY}*{Section 5}. The computation is similar to the case where $l=1$ and the multiplicity is $(D\cdot\beta)^{l}$ and we leave more combinatorial details to the readers.
\end{proof}

\begin{proposition}\label{prop-several-neg-1}
Let $(X,D)$ be a smooth pair. Given a curve class $\beta$, let $k_i\in \mathbb Z_{>0}$ for $i\in \{1,\ldots, l+1\}$  such that
\[
D\cdot \beta=k_{l+1}-\sum_{i=1}^l k_i\geq 0.
\]
Let $a\in \mathbb Z_{\geq 0}$.
Then we have the following relation:
\begin{align}\label{identity-several-neg-1}
\langle  [1]_{-k_1},\cdots, [1]_{-k_l}, [\gamma]_{k_{l+1}} \bar{\psi}^a\rangle_{0,l+1,\beta}^{(X,D)}=\langle  [D]_0,\cdots, [D]_0, [\gamma]_{D\cdot \beta} \bar{\psi}^a\rangle_{0,l+1,\beta}^{(X,D)},
\end{align}
where $\gamma \in H^*(D)$ when $D\cdot\beta>0$; $\gamma \in H^*(X)$ on the RHS of (\ref{identity-several-neg-1}) and $\gamma=\iota^*\gamma \in \iota^*H^*(X)\subset H^*(D)$ on the LHS of (\ref{identity-several-neg-1}) when $D\cdot\beta=0$.
\end{proposition}

\begin{proof}
When $D$ is nef, this is \cite{You22}*{Proposition 4.6}. Now we remove the assumption that $D$ is nef.

We prove it by induction on the number of markings with negative contact orders. The induction does not follow directly from that of \cite{You22}*{Proposition 4.6} and it requires extra care because $D$ is not nef.

{\bf (1) Base case: $l=1$:} we need to prove that
\begin{align}\label{iden-neg-l=1}
\langle  [1]_{-k_1}, [\gamma]_{k_{2}} \bar{\psi}^a\rangle_{0,2,\beta}^{(X,D)}=\langle [D]_0,  [\gamma]_{D\cdot \beta} \bar{\psi}^a\rangle_{0,2,\beta}^{(X,D)}.
\end{align}
For the LHS of (\ref{iden-neg-l=1}), we apply the graph sum definition of relative invariants with negative contact orders in \cite{FWY}. There is one rubber moduli for the only negative contact marking. Suppose that the last marking $p_{2}$ is not distributed to the rubber, then we have the following degeneration contributions:
\begin{align*}
\sum_{\mathfrak G\in \mathcal B_\Gamma}\frac{\prod_{e\in E}d_e}{|\Aut(\mathfrak G)|}\sum \langle  |\,|[1]_{k_1},\eta\rangle^{\sim} \langle  \check{\eta}, [\gamma]_{k_2}\psi^a \rangle^{\bullet, (X,D)}.
\end{align*}

Now we analyse the rubber integral:
\[
\langle  |\,|[1]_{k_1},\eta\rangle^{\sim}. 
\]
Using the same argument as in the proof of Lemma \ref{lemma-several-neg-0}, this rubber integral is zero. We conclude that the last marking $p_2$ is also distributed to the rubber moduli.

Therefore, the invariant is
\begin{align*}
\langle [1]_{-k_1},  [\gamma]_{k_{2}} \bar{\psi}^a\rangle_{0,l+1,\beta}^{(X,D)}=\sum_{\mathfrak G\in \mathcal B_\Gamma}\frac{\prod_{e\in E}d_e}{|\Aut(\mathfrak G)|}\sum \langle [\gamma]_{k_2}\bar{\psi}^a |\,|[1]_{k_1},\eta\rangle^{\sim} \langle  \check{\eta}\rangle^{\bullet, (X,D)}.
\end{align*}

For the RHS of (\ref{iden-neg-l=1}), we consider the degeneration of $X$ to the normal cone of $D$ and apply the degeneration formula. The first marking has an insertion $[D]_0$, so it is distributed to the $(Y,D_0+D_\infty)$-side. The rigidification lemma of \cite{MP}*{Lemma 2} says that
\[
\pi_*\left(\on{ev}_{2,*}{[D]_0}\cap [\bM(Y,D_0\cup D_\infty)]^{\on{vir}}\right)=[\bM^\sim(Y,D_0\cup D_\infty)]^{\on{vir}}.
\]
Then the first marking also becomes the identity class. The rest of the analysis is similar to the analysis of the LHS of (\ref{iden-neg-l=1}). The degeneration formula is
\begin{align*}
\langle  [D]_0, [\gamma]_{D\cdot \beta} \bar{\psi}^a\rangle_{0,l+1,\beta}^{(X,D)}=\sum_{\mathfrak G\in \mathcal B_\Gamma}\frac{\prod_{e\in E}d_e}{|\Aut(\mathfrak G)|}\sum \langle [\gamma]_{k_2-k_1}\bar{\psi}^a |\,|[1]_{0},\eta\rangle^{\sim} \langle  \check{\eta}\rangle^{\bullet, (X,D)}.
\end{align*}

Now, the only difference between the LHS of (\ref{iden-neg-l=1}) and the RHS of (\ref{iden-neg-l=1}) is the contact orders of two markings. We pushforward the rubber moduli spaces to the moduli space $\overline{M}(D)$ of stable maps to $D$. The LHS of (\ref{iden-neg-l=1}) and the RHS of (\ref{iden-neg-l=1}) are the same because the genus zero double ramification cycle is trivial.

{\bf (2) Induction.}

The case where $a=0$ was proved in Lemma \ref{identity-several-neg-0}. We now assume that $a\geq 1$.

Suppose that (\ref{identity-several-neg-1}) is true when $l=N\geq 1$. Let $l=N+1$. 
%We can add an additional insertion with the divisor class $[D]_0$ and prove the following\[\langle  [D]_0, [1]_{-k_1},\cdots, [1]_{-k_{N+1}}, [\gamma]_{k_{N+2}} \bar{\psi}^a\rangle_{0,N+3,\beta}^{(X,D)}=\langle [D]_0, [D]_0,\cdots, [D]_0, [\gamma]_{D\cdot \beta} \bar{\psi}^a\rangle_{0,N+3,\beta}^{(X,D)}.\]
We apply the topological recursion relation:
\begin{align*}
&\langle [1]_{-k_1},[1]_{-k_2},\cdots, [1]_{-k_{N+1}}, [\gamma]_{k_{N+2}} \bar{\psi}^{a}\rangle_{0,N+2,\beta}^{(X,D)}\\
=&\sum \langle [\gamma]_{k_{N+2}} \bar{\psi}^{a-1}, \prod_{j\in S_1}[1]_{-k_j},\tilde {T}_{i,k}\rangle_{0,\beta_1,2+|S_1|}^{(X,D)} \langle \tilde {T}_{-i}^k, \prod_{j\in S_2}[1]_{-k_j},[1]_{-k_1},[1]_{-k_2} \rangle_{0,\beta_2,3+|S_2|}^{(X,D)},
\end{align*}
where the sum is over all $\beta_1+\beta_2=\beta$, all indices $i,k$ of basis and $S_1, S_2$ disjoint sets with $S_1\sqcup S_2=\{3,\ldots,N+1\}$.

By Lemma \ref{lemma-several-neg-0}, if $i<0$, the invariant $$ \langle \tilde {T}_{-i}^k, \prod_{j\in S_2}[1]_{-k_j},[1]_{-k_1},[1]_{-k_2} \rangle_{0,\beta_2,3+|S_2|}^{(X,D)}$$ is zero when $D\cdot\beta_2<0$. When $D\cdot\beta_2\geq 0$, we have
\begin{align*}
&\langle \tilde {T}_{-i}^k, \prod_{j\in S_2}[1]_{-k_j},[1]_{-k_1},[1]_{-k_2} \rangle_{0,\beta_2,3+|S_2|}^{(X,D)}\\
=&\langle \tilde {T}_{-i-k_1-\sum_{j\in S_2}k_j}^k, \prod_{j\in S_2}[D]_{0},[D]_{0},[D]_{0} \rangle_{0,\beta_2,3+|S_2|}^{(X,D)}\\
%=&(D\cdot\beta_2)\langle \tilde {T}_{-i-k_1-\sum_{j\in S_2}k_j}^k, \prod_{j\in S_2}[D]_{0},[D]_{0} \rangle_{0,\beta_2,2+|S_2|}^{(X,D)}\\
%=&(D\cdot\beta_2)\langle \tilde {T}_{-i}^k, \prod_{j\in S_2}[1]_{-k_j},[1]_{-k_1-k_2}\rangle_{0,\beta_2,2+|S_2|}^{(X,D)}\\
=&\langle \tilde {T}_{-i}^k, \prod_{j\in S_2}[1]_{-k_j},[1]_{-k_1-k_2},[D]_{0} \rangle_{0,\beta_2,3+|S_2|}^{(X,D)}.
\end{align*}
If $i>0$, we also have the same identity by Lemma \ref{lemma-several-neg-0-}. In other words, we always have
\[
 \langle \tilde {T}_{-i}^k, \prod_{j\in S_2}[1]_{-k_j},[1]_{-k_1},[1]_{-k_2} \rangle_{0,\beta_2,3+|S_2|}^{(X,D)}=\langle \tilde {T}_{-i}^k, \prod_{j\in S_2}[1]_{-k_j},[1]_{-k_1-k_2},[D]_{0} \rangle_{0,\beta_2,3+|S_2|}^{(X,D)}.
\]

Combining these together, we have:
\begin{align*}
    &\sum \langle [\gamma]_{k_{N+2}} \bar{\psi}^{a}, \prod_{j\in S_1}[1]_{-k_j}, \tilde {T}_{i,k}\rangle_{0,\beta_1,2+|S_1|}^{(X,D)} \langle \tilde {T}_{-i}^k, \prod_{j\in S_2}[1]_{-k_j}, [1]_{-k_1},[1]_{-k_2} \rangle_{0,\beta_2,3+|S_2|}^{(X,D)}\\
    =&\sum \langle [\gamma]_{k_{N+2}} \bar{\psi}^{a}, \prod_{j\in S_1}[1]_{-k_j}, \tilde {T}_{i,k}\rangle_{0,\beta_1,2+|S_1|}^{(X,D)} \langle \tilde {T}_{-i}^k,\prod_{j\in S_2}[1]_{-k_j}, [1]_{-k_1-k_2},[D]_0 \rangle_{0,\beta_2,3+|S_2|}^{(X,D)}\\
    =&\langle [D]_0, [1]_{-k_1-k_2}, \prod_{j\in\{3,\ldots,N+1\}}[1]_{-k_j},[\gamma]_{k_{N+2}} \bar{\psi}^{a}\rangle_{0,N+2,\beta}^{(X,D)},\\
\end{align*}
where the third line is the topological recursion relation. We have the identity

\begin{align*}
&\langle [1]_{-k_1},\cdots, [1]_{-k_{N+1}}, [\gamma]_{k_{N+2}} \bar{\psi}^{a}\rangle_{0,N+2,\beta}^{(X,D)}\\
=&\langle [D]_0,[1]_{-k_1-k_2},[1]_{-k_3},\cdots, [1]_{-k_{N+1}}, [\gamma]_{k_{N+2}} \bar{\psi}^{a}\rangle_{0,N+2,\beta}^{(X,D)}.
\end{align*}
Run the above argument multiple times for other markings with insertions $[1]_{k_i}$ to trade markings with negative contact orders with markings with insertions $[D]_0$. In the end, it reduces to the base case. This concludes the proof.

\end{proof}

Because we do not assume that $D$ is nef, we need a couple of more identities to compute invariants with $D\cdot\beta<0$.

It is possible that $D\cdot\beta<0$, but $k_{l+1}\geq 0$. Then we have a slightly different identity as follows.

\begin{proposition}\label{prop-several-neg-4}
Let $(X,D)$ be a smooth pair. Given a curve class $\beta$, let $k_i\in \mathbb Z_{>0}$ for $i\in \{1,\ldots, l\}$ such that
\[
D\cdot \beta=k_{l+1}-\sum_{i=1}^l k_i< 0 \text{ and } k_{l+1}\geq  0.
\]
Let $a\geq 0$, then we have the following relation:
\begin{align}\label{identity-several-neg-4}
\langle  [1]_{-k_1},\cdots, [1]_{-k_l}, [\gamma]_{k_{l+1}} \bar{\psi}^{a}\rangle_{0,l+1,\beta}^{(X,D)}=\langle  [D]_0,\cdots, [D]_0, [1]_0, [\gamma]_{D\cdot \beta} \bar{\psi}^{a}\rangle_{0,l+1,\beta}^{(X,D)},
\end{align}
where $\gamma \in H^*(D)$ when $k_{l+1}>0$; $\gamma \in H^*(X)$ on the LHS of (\ref{identity-several-neg-4}) and $\gamma=\iota^*\gamma \in \iota^*H^*(X)\subset H^*(D)$ on the RHS of (\ref{identity-several-neg-4})  when $k_{l+1}=0$.
\end{proposition}

\begin{proof}
We first note that the last marking of the LHS of (\ref{identity-several-neg-4}) has non-negative contact and the last marking of the RHS of (\ref{identity-several-neg-4}) has negative contact. We have a marking with an insertion $[1]_0$ on the LHS, so the virtual dimension constraints are satisfied on both sides. 

When $a=0$,  the LHS of (\ref{identity-several-neg-4}) is zero by Lemma \ref{lemma-several-neg-0}. The RHS of (\ref{identity-several-neg-4}) is zero by the string equation. Therefore, both sides are zero. 

If $a\geq1$, we can also write the RHS of (\ref{identity-several-neg-4}) as
\[
\langle  [D]_0,\cdots, [D]_0, [\gamma]_{D\cdot \beta} \bar{\psi}^{a-1}\rangle_{0,l,\beta}^{(X,D)}.
\]

Again, we prove it by induction. When $l=1$, there are two markings. We first consider the LHS of (\ref{identity-several-neg-4}). Using the definition, both markings are in the same rubber. Applying the string equation, the rubber invariant becomes
\[
\langle \prod_i \eta_i, \gamma \psi^{a-1}\rangle_{0,1+|\eta|,\pi_*\beta}^D,
\]
where $\pi: Y:=\mathbb P(\mathcal O_D\oplus N_D)\rightarrow D$ is the projection. This agrees with the degeneration contributions for the RHS (\ref{identity-several-neg-4}). Then we proved the base case. The rest of the induction is identical to the proof of Proposition \ref{prop-several-neg-1}.

\end{proof}

\begin{proposition}\label{prop-several-neg-3}
Let $(X,D)$ be a smooth pair. Given a curve class $\beta$, let $k_i\in \mathbb Z_{>0}$ for $i\in \{1,\ldots, l\}$ such that
\[
D\cdot \beta=k_{l+1}-\sum_{i=1}^l k_i< 0 \text{ and } k_{l+1}< 0.
\]
Then, we have the following relation:
\begin{align}\label{identity-several-neg-3}
\langle  [1]_{-k_1},\cdots, [1]_{-k_l}, [\gamma]_{k_{l+1}} \bar{\psi}^a\rangle_{0,l+1,\beta}^{(X,D)}=\langle  [D]_0,\cdots, [D]_0, [\gamma]_{D\cdot \beta} \bar{\psi}^a\rangle_{0,l+1,\beta}^{(X,D)},
\end{align}
where $\gamma \in H^*(D)$ and $a\geq 0$.
\end{proposition}
\begin{proof}

For $a\geq 0$, we have the following identity:
\[
\langle  [1]_{-k_1},\cdots, [1]_{-k_l}, [\gamma]_{k_{l+1}} \bar{\psi}^a\rangle_{0,l+1,\beta}^{(X,D)}=\langle  [1]_{-k_1},\cdots, [1]_{-k_l}, [1]_{k_{l+1}-b},[\gamma]_{b} \bar{\psi}^{a+1}\rangle_{0,l+2,\beta}^{(X,D)},
\]
for a positive integer $b$ (note that $k_{l+1}-b<0$). This is because of the following. Both invariants are genus zero relative invariants with $(l+1)$ markings of negative contact orders. Furthermore, in both cases, there is only one rubber moduli, and all markings are distributed to the rubber. On the RHS, there are $(l+1)$ markings that have insertion $[1]$ (after pushing forward to the moduli space of stable maps to $D$.) The class $\mathfrak c_{\Gamma}$ is of degree $2l-2$. When applying the string equation, at least one of $[1]$ should cancel with the $\psi$-class from the last marking. On the RHS, we have one extra $[1]$ that cancels with the one extra $\bar{\psi}$. Therefore, the invariants are the same.

By Proposition \ref{prop-several-neg-4}, we have
\begin{align*}
&\langle  [1]_{-k_1},\cdots, [1]_{-k_l}, [\gamma]_{k_{l+1}} \bar{\psi}^a\rangle_{0,l+1,\beta}^{(X,D)}\\
=&\langle  [1]_{-k_1},\cdots, [1]_{-k_l}, [1]_{k_{l+1}-b},[\gamma]_{b} \bar{\psi}^{a+1}\rangle_{0,l+2,\beta}^{(X,D)}\\
=&\langle  [D]_{0},\cdots, [D]_{0}, [1]_{0},[\gamma]_{D\cdot\beta} \bar{\psi}^{a+1}\rangle_{0,l+2,\beta}^{(X,D)}\\
=& \langle  [D]_{0},\cdots, [D]_{0}, [\gamma]_{D\cdot\beta} \bar{\psi}^{a}\rangle_{0,l+1,\beta}^{(X,D)}.
\end{align*}

\end{proof}

We write
\begin{equation}
\begin{split}
g(y):=\sum_{d\geq 2} g_d(y),
\end{split}
\end{equation}
where 
\begin{equation*}
\begin{split}
g_d(y):=\sum_{\substack{\beta\in \NE(X): D\cdot\beta=d}}\left[\sum_{l\geq 0,\beta^\prime\in \NE(X)}\frac{1}{l!}\langle [\on{pt}]\psi^{d-2},\tau_D,\ldots,\tau_D\rangle_{0,1+l,\beta^\prime}^X y^{\beta^\prime}\right]_{y^\beta}y^\beta (d-1)!.
\end{split}
\end{equation*}

Then we write the relative mirror map (\ref{mirror-map}) as
\begin{equation}
\begin{split}
\tau(y):=\sum_{i=1}^{\mathrm r}p_i \log y_i +\sum_{d\geq 2}g_d(y)[1]_{-d}.
\end{split}
\end{equation}

Recall that
\begin{align*}
J_{(X,D)}(\tau(y),z)=\sum_{\beta\in\NE(X),l\geq 0}\frac{1}{l!}\langle \sum_{d\geq 2}g_d(y)[1]_{-d},\ldots,\sum_{d\geq 2}g_d(y)[1]_{-d}, [\gamma]_{k_0} \bar{\psi}^a\rangle_{0,l+1,\beta}^{(X,D)}[\gamma^\vee]_{-k_0} y^\beta.
\end{align*}

We consider the ambient $J$-function, that is, the class $\gamma$ is pullback from $H^*(X)$. The invariants are computed in Propostion \ref{prop-several-neg-1}, Proposition \ref{prop-several-neg-3} and Proposition \ref{prop-several-neg-4}.
\begin{itemize}
\item When $D\cdot\beta\geq 0$, $k_0>0$, we have:
\[
\langle  [1]_{-k_1},\cdots, [1]_{-k_l}, [\gamma]_{k_{0}} \bar{\psi}^a\rangle_{0,l+1,\beta}^{(X,D)}=\langle  [D]_0,\cdots, [D]_0, [\gamma]_{D\cdot \beta} \bar{\psi}^a\rangle_{0,l+1,\beta}^{(X,D)}.
\]
\item When $D\cdot\beta<0$, $k_0<0$, we have:
\[
\langle  [1]_{-k_1},\cdots, [1]_{-k_l}, [\gamma]_{k_{0}} \bar{\psi}^a\rangle_{0,l+1,\beta}^{(X,D)}=\langle  [D]_0,\cdots, [D]_0, [\gamma]_{D\cdot \beta} \bar{\psi}^a\rangle_{0,l+1,\beta}^{(X,D)}.
\]
\item When $D\cdot\beta<0$, $k_0\geq 0$, we have:
\[
\langle  [1]_{-k_1},\cdots, [1]_{-k_l}, [\gamma]_{k_{0}} \bar{\psi}^{a}\rangle_{0,l+1,\beta}^{(X,D)}=\langle  [D]_0,\cdots, [D]_0, [1]_0, [\gamma]_{D\cdot \beta} \bar{\psi}^{a}\rangle_{0,l+1,\beta}^{(X,D)}.
\]
\end{itemize}

We consider the product of the $J$-function with $[1]_{-D\cdot\beta+k_0}=[1]_{\sum_{i=1}^l k_i}$.  When $D\cdot\beta\geq 0, k_0>0$ or $D\cdot\beta<0, k_0<0$, we have:
\begin{align*}
&\langle  [1]_{-k_1},\cdots, [1]_{-k_l}, [\gamma]_{k_{0}} \bar{\psi}^a\rangle_{0,l+1,\beta}^{(X,D)}[\gamma^\vee]_{-k_0}\cdot [1]_{-D\cdot\beta+k_0}\\
=&\langle  [D]_0,\cdots, [D]_0, [\gamma]_{D\cdot \beta} \bar{\psi}^a\rangle_{0,l+1,\beta}^{(X,D)}[\gamma^\vee]_{-D\cdot\beta};
\end{align*}
when $D\cdot\beta<0$, $k_0\geq 0$, we have:
\begin{align*}
&\langle  [1]_{-k_1},\cdots, [1]_{-k_l}, [\gamma]_{k_{0}} \bar{\psi}^a\rangle_{0,l+1,\beta}^{(X,D)}[\gamma^\vee]_{-k_0}\cdot [1]_{-D\cdot\beta+k_0}\\
=&\langle  [D]_0,\cdots, [D]_0, [1]_0,[\gamma]_{D\cdot \beta} \bar{\psi}^a\rangle_{0,l+1,\beta}^{(X,D)}[\gamma^\vee\cup D]_{-D\cdot\beta}\\
=&\langle  [D]_0,\cdots, [D]_0, [\gamma]_{D\cdot \beta} \bar{\psi}^{a-1}\rangle_{0,l,\beta}^{(X,D)}[\gamma^\vee\cup D]_{-D\cdot\beta},
\end{align*}
where, as computed in \cite{FWY}*{Section 7.1},
\begin{align*}
[\gamma^\vee]_{-k_0}\cdot [1]_{-D\cdot\beta+k_0}=
\left\{
\begin{array}{cc}
\,  [\gamma^\vee]_{-D\cdot\beta} & \text{if }, D\cdot\beta\geq 0, k_0>0 \text{ or }D\cdot\beta<0, k_0<0\\ 
\, [\gamma^\vee\cup D]_{-D\cdot\beta} & \text{if } D\cdot\beta<0, k_0\geq 0.
\end{array}
\right.
\end{align*}

Hence, the insertion $\sum_{d\geq 2}g_d(y)[1]_{-d}$ has the same effect as the insertion $g(y)[D]_0$. The divisor class $[D]_0$ can be removed using the divisor equation. 
Hence, we can write the relative mirror map (\ref{mirror-map}) as
\begin{align}\label{iden-rel-mirror-map}
\sum_{i=1}^{\mathrm r} p_i\log q_i=\sum_{i=1}^{\mathrm r} p_i\log y_i+g(y)D.
\end{align}

Now we consider the invariants that will appear in the theta function $\vartheta_1$. We have an extra marking with insertion $[1]_1$.
\begin{proposition}\label{prop-several-neg-2}
Let $(X,D)$ be a smooth pair. Given a curve class $\beta$, let $k_i\in \mathbb Z_{>0}$ for $i\in \{1,\ldots, l+1\}$ such that
\[
D\cdot \beta-1=k_{l+1}-\sum_{i=1}^l k_i\neq 0.
\]
Then we have the following relation.
\begin{equation}\label{identity-several-neg-2}
\begin{split}
&\langle [1]_1, [1]_{-k_1},\cdots, [1]_{-k_l},  [\on{pt}]_{k_{l+1}} \rangle_{0,l+2,\beta}^{(X,D)}\\
=&
\left\{
\begin{array}{cc}
(D\cdot \beta-1)^l\langle  [1]_1, [\on{pt}]_{D\cdot \beta-1}  \rangle_{0,2,\beta}^{(X,D)} & \text{if } D\cdot\beta>1\\
0 & \text{if }D\cdot\beta \leq 1 \text{ and }D\cdot\beta\neq 0.
\end{array}
\right.
\end{split}
\end{equation}

\end{proposition}

\begin{proof}
When $D$ is nef, this is \cite{You22}*{Proposition 4.7}. Now we remove the assumption that $D$ is nef. The proof without the nefness assumption is similar to the proof in Lemma \ref{lemma-several-neg-0}. If $D\cdot\beta\neq 0$, the first marking cannot be distributed to the rubber because it will make the rubber invariants equal zero. We also know that there is only one rubber moduli, and the last marking is distributed to this rubber. This implies that the degeneration formula is:
\begin{align*}
&\langle [1]_1, [1]_{-k_1},\cdots, [1]_{-k_l},  [\on{pt}]_{k_{l+1}} \rangle_{0,l+2,\beta}^{(X,D)}\\
=&\sum_{\mathfrak G\in \mathcal B_\Gamma}\frac{\prod_{e\in E}d_e}{|\Aut(\mathfrak G)|}\sum_\eta \langle [1]_1,\check{\eta}\rangle^{\bullet,\mathfrak c_\Gamma, (X,D)}\langle \eta,[1]_{k_1},\cdots, [1]_{k_l},| \, | [\on{pt}]_{k_{l+1}} \rangle^{\sim,\mathfrak c_{\Gamma}}. 
\end{align*}

The rest of the proof is similar to the proof of Proposition \ref{prop-several-neg-1}.
\end{proof}

When $D\cdot\beta=0$, all the markings are distributed to the rubber. Moreover, the rubber invariant will be zero unless $\pi_*(\beta)=0$, where $\pi: \mathbb P(\mathcal O_D\oplus N_D)\rightarrow D$ is the projection map. Then we must have $\beta=0$.
The degree zero invariants are already computed in \cite{You22} without assuming that $D$ is nef.
\begin{prop}[=\cite{You22}*{Proposition 4.9}]\label{prop-degree-zero}
\begin{align}\label{identity-zero-several-neg}
\langle [\on{pt}]_{k_{l+1}}, [1]_1,[1]_{-k_1},\cdots, [1]_{-k_l}\rangle_{0,l+2,0}^{(X,D)}=(-1)^{l-1},    
\end{align}
where $k_1, \ldots, k_l$ are positive integers and
\[
1+k_{l+1}=k_1+\cdots +k_l.
\]
\end{prop}

\subsection{The proper potential from the relative mirror theorem}

We can combine the calculations in the previous sections to obtain a formula for the proper potential $W:=\vartheta_1$ when $D$ is not nef. 
\begin{theorem}\label{thm-proper-potential}
Given a smooth log Calabi--Yau pair $(X,D)$, the proper LG potential is
\begin{align}\label{iden-proper-potential}
W:=x+\sum_{n=1}^{\infty}\sum_{\beta:D\cdot\beta=n+1}n\left\langle [1]_1,[\on{pt}]_{n}\right\rangle_{0,2,\beta}^{(X,D)}t^\beta x^{-n}=x\exp(g(y(q))),
\end{align}
where $y(q)$ is the inverse of the relative mirror map (\ref{iden-rel-mirror-map}) and 
\[
g(y)=\sum_{\substack{\beta\in \NE(X) \\ d=D\cdot\beta\geq 2}}\left[\sum_{l\geq 0,\beta^\prime\in \NE(X)}\frac{1}{l!}\langle [\on{pt}]\psi^{d-2},\tau_D,\ldots,\tau_D\rangle_{0,1+l,\beta^\prime}^X y^{\beta^\prime}\right]_{y^\beta}y^\beta (d-1)!.
\]
\end{theorem}

The remainder of this section is to prove Formula (\ref{iden-proper-potential}). Using the identities in Section \ref{sec:iden-rel-GW}, we can compute the proper potential $W$. The computation is similar to that of \cite{You22}. We will emphasise the difference: curve class $\beta$ with $D\cdot\beta<0$ and the mirror map in $D$.

The $I$-function is the $S$-extended $I$-function (\ref{rel-I-func-extended}) with $S=\{1\}$.

We extract the sum over the coefficients of $x_1z^{-1}$ of the $J$-function and the $I$-function that takes values in $[1]_{n}$ for $n\geq 1$. The coefficient of the $J$-function is
\begin{align*}
&\sum_{\beta: D\cdot \beta\geq 1,n\geq 1}\sum_{l\geq 0}\langle [1]_1,\tau(y),\cdots, \tau(y),[\on{pt}]_{n}\rangle_{0,\beta,l+2}^{(X,D)}\frac{1}{l!} \\
 +& \sum_{\beta: D\cdot \beta=0}\sum_{n\geq 1, l>0}\langle[1]_1,\tau(y),\cdots,\tau(y),[\on{pt}]_n \rangle_{0,\beta,l+2}^{(X,D)}\frac{1}{l!}. 
\end{align*}
Using Proposition \ref{prop-several-neg-2} and Proposition \ref{prop-degree-zero}, it becomes
\[
\sum_{\beta: D\cdot \beta=n+1,n\geq 1}\langle [1]_1,[\on{pt}]_{n}\rangle_{0,\beta,2}^{(X,D)}y^\beta\exp\left(g(y)(D\cdot\beta-1)\right)  -\exp\left(-g(y)\right)+1.
\]
Using the change of variables
\[
\sum_{i=1}^{\mathrm r}\log q_i=\sum_{i=1}^{\mathrm r}\log y_i+g(y)D,
\]
we can write it as follows:
\[
\exp\left(-g(y)\right)\sum_{\beta: D\cdot \beta=n+1,n\geq 1}\langle [1]_1,[\on{pt}]_{n}\rangle_{0,\beta,2}^{(X,D)}q^\beta -\exp\left(-g(y)\right)+1.
\]

The corresponding coefficient of the $I$-function is
\[
\sum_{\substack{\beta\in \NE(X) \\ d=D\cdot\beta\geq 2}}\left[\sum_{l\geq 0,\beta^\prime\in \NE(X)}\frac{1}{l!}\langle [\on{pt}]\psi^{d-2},\tau_D,\ldots,\tau_D\rangle_{0,1+l,\beta^\prime}^X y^{\beta^\prime}\right]_{y^\beta}y^\beta \frac{(d)!}{d-1}.
\]

We have the following relation from the relative mirror theorem:
\begin{equation}\label{iden-J-I-coeff}
\begin{split} 
\exp\left(-g(y)\right)\sum_{\beta: D\cdot \beta=n+1,n\geq 1}\langle [1]_1,[\on{pt}]_{n}\rangle_{0,\beta,2}^{(X,D)}q^\beta  -\exp\left(-g(y)\right)+1\\
=\sum_{\substack{\beta\in \NE(X) \\ d=D\cdot\beta\geq 2}}\left[\sum_{l\geq 0,\beta^\prime\in \NE(X)}\frac{1}{l!}\langle [\on{pt}]\psi^{d-2},\tau_D,\ldots,\tau_D\rangle_{0,1+l,\beta^\prime}^X y^{\beta^\prime}\right]_{y^\beta}y^\beta \frac{(d)!}{d-1}
\end{split}
\end{equation}

Write $D=\sum_{i=1}^{\mathrm r} m_ip_i$ for some $m_i\in \mathbb Z_{\geq 0}$. We apply the differential operator
\[
\Delta_D=\sum_{i=1}^{\mathrm r}m_iy_i\frac{\partial}{\partial y_i}-1
\]
to (\ref{iden-J-I-coeff}), we have
\begin{align*}
&\left(\sum_{\substack{\beta\in \NE(X) \\ d=D\cdot\beta\geq 2}}\left[\sum_{l\geq 0,\beta^\prime\in \NE(X)}\frac{1}{l!}\langle [\on{pt}]\psi^{d-2},\tau_D,\ldots,\tau_D\rangle_{0,1+l,\beta^\prime}^X y^{\beta^\prime}\right]_{y^\beta}y^\beta (d)!\right)\\
&\cdot\exp(-g(y))\left(1+\sum_{n=1}^{\infty}\sum_{\beta:D\cdot\beta=n+1}n\left\langle [1]_1,[\on{pt}]_{n}\right\rangle_{0,2,\beta}^{(X,D)}q^\beta\right)\\
=&\sum_{\substack{\beta\in \NE(X) \\ d=D\cdot\beta\geq 2}}\left[\sum_{l\geq 0,\beta^\prime\in \NE(X)}\frac{1}{l!}\langle [\on{pt}]\psi^{d-2},\tau_D,\ldots,\tau_D\rangle_{0,1+l,\beta^\prime}^X y^{\beta^\prime}\right]_{y^\beta}y^\beta (d)!.
\end{align*}

Therefore, we have
\[
1+\sum_{n=1}^{\infty}\sum_{\beta:D\cdot\beta=n+1}n\left\langle [1]_1,[\on{pt}]_{n}\right\rangle_{0,2,\beta}^{(X,D)}q^\beta=\exp(g(y(q))).
\]

\begin{remark}
The main difference from the case where $D$ is nef is that $g(y)$ is defined by absolute invariants of $X$ with the mirror map that includes contributions from certain curve countings in $D$. The difference here is also the same as the difference between the regularized quantum periods and the classical periods in Theorem \ref{thm-classical-quantum}.
\end{remark}

\section{The Lagrange inversion and Bell polynomial identities}\label{sec:bell-poly}

In this section, we use the Lagrange inversion for Laurent series and Bell polynomial identities, which were also studied in \cite{GRZZ} and \cite{BGL} to show that the proper potential and the quantum period with the mirror map in $D$ are equivalent data.

Recall that the proper potential is obtained in Theorem \ref{thm-proper-potential}:
\begin{align*}
W:=x+\sum_{n=1}^{\infty}\sum_{\beta:D\cdot\beta=n+1}n\left\langle [1]_1,[\on{pt}]_{n}\right\rangle_{0,2,\beta}^{(X,D)}t^\beta x^{-n}=x\exp(g(y(q))),
\end{align*}
where $W$ is a function of $q$ under the identification $q^\beta=t^\beta x^{-D\cdot\beta}$. We can set $t=1$ and think of it as a function of $x$. Now we make a change of variables $x\rightarrow x^{-1}$, then we can write 
\begin{align*}
W:=x^{-1}+\sum_{n=1}^{\infty}\sum_{\beta:D\cdot\beta=n+1}n\left\langle [1]_1,[\on{pt}]_{n}\right\rangle_{0,2,\beta}^{(X,D)} x^{n}=x^{-1}\exp(g(y(x))).
\end{align*}

We also set $\omega=y^{-1}$. Then, the relative mirror map (\ref{iden-rel-mirror-map}) becomes
\[
x=\omega^{-1}\exp \left( g(y)\right).
\]
\[
W(\omega^{-1}\exp \left( g(y)\right))=\omega.
\]

We need to use the following Lagrange inversion theorem.
\begin{lemma}
Let $$f(x)=x^{-1}+\sum_{k\geq 0}f_k x^k$$ be a Laurent series with only simple poles. Let $g(\omega)$ be the composition inverse: $f(g(\omega))=\omega$. Then $g(\omega)$ is a power series in $\omega^{-1}$ given by
\[
g(\omega)=\sum_{k>0}\frac{\omega^{-k}}{k}[f^k]_{x^{-1}}.
\]
\end{lemma}

The Lagrange inversion theorem for the Laurent series implies that
\[
\omega^{-1}\exp \left( g(y)\right)=\sum_{k> 0}\frac{\omega^{-k}}{k} [W^k]_{x^{-1}}.
\]
Furthermore, we have the following identity from \cite{GRZZ}*{Proposition B.12} and \cite{BGL}*{Lemma 10.5} by the Bell polynomial identities.
\begin{lemma}
Let $f(x)=1+\sum_{k\geq 1} f_k x^k$ be a power series. Then the following equality holds
\[
\exp\left(\sum_{k>0} \frac{1}{k} [f^k]_{x^k} y^k\right)=\sum_{k>0} \frac{1}{k} [f^k]_{x^{k-1}}y^{k-1}.
\]
\end{lemma}

Therefore, letting $f=x\cdot W$, we have
\begin{align*}
\exp \left( g(y)\right)&=w\sum_{k> 0}\frac{\omega^{-k}}{k} [W^k]_{x^{-1}}\\
&=\sum_{k> 0}\frac{y^{k-1}}{k} [x^kW^k]_{x^{k-1}}\\
&=\exp\left(\sum_{k> 0}\frac{1}{k}[x^k W^k]_{x^k}y^k\right)\\
&=\exp\left(\sum_{k> 0}\frac{1}{k}[W^k]_{x^0}y^k\right).
\end{align*}
Therefore, we have
\[
g(y)=\sum_{k> 0}\frac{1}{k}[W^k]_{x^0}y^k,
\]
where
\[
g(y)=\sum_{d\geq 2}\sum_{\substack{\beta\in \NE(X) \\ D\cdot\beta=d}}\left[\sum_{l\geq 0,\beta^\prime\in \NE(X)}\frac{1}{l!}\langle [\on{pt}]\psi^{d-2},\tau_D,\ldots,\tau_D\rangle_{0,1+l,\beta^\prime}^X y^{\beta^\prime}\right]_{y^\beta}y^\beta (d-1)!.
\]
This implies that
\[
\frac{1}{d}[W^d]_{x^0}=\sum_{\substack{\beta\in \NE(X): D\cdot\beta=d}}\left[\sum_{l\geq 0,\beta^\prime\in \NE(X)}\frac{1}{l!}\langle [\on{pt}]\psi^{d-2},\tau_D,\ldots,\tau_D\rangle_{0,1+l,\beta^\prime}^X y^{\beta^\prime}\right]_{y^\beta}y^\beta (d-1)!.
\]

In Theorem \ref{thm-classical-quantum}, we computed $[W^k]_{\vartheta_0}$. It looks slightly different from $[W^k]_{x^0}$. However, it is not difficult to see that they are the same: by the definition of orbifold theta functions in \cite{You24} and TRR for relative Gromov--Witten theory in \cite{FWY}, we have
\[
[W^k]_{x^0}=\sum_{\beta}\langle [1]_1,\ldots,[1]_1,[1]_{\mathbf b}, [\on{pt}]_{-\mathbf b}\bar{\psi}^{D\cdot\beta-1}\rangle^{(X,D)}_{0,D\cdot\beta+2,\beta},
\]
where the last two markings are mid-age markings. By a degeneration argument, we have the following equality:
\begin{align*}
&\langle [1]_1,\ldots,[1]_1,[1]_{\mathbf b}, [\on{pt}]_{-\mathbf b}\bar{\psi}^{D\cdot\beta-1}\rangle^{(X,D)}_{0,D\cdot\beta+2,\beta}\\
=&\langle [1]_1,\ldots,[1]_1, [\on{pt}]_{0}\bar{\psi}^{D\cdot\beta-2}\rangle^{(X,D)}_{0,D\cdot\beta+1,\beta}.
\end{align*}
Therefore,
\[
[W^k]_{x^0}=[W^k]_{\vartheta_0}.
\]
Hence, we have another proof of Theorem \ref{thm-classical-quantum} from Theorem \ref{thm-proper-potential}:
\begin{align*}
\pi_W(t)&=\sum_{d\geq 0}[W^d]_{\vartheta_0}\\
&=1+\sum_{\substack{\beta\in \NE(X) \\ d=D\cdot\beta\geq 2}}d\left[\sum_{l\geq 0,\beta^\prime\in \NE(X)}\frac{1}{l!}\langle [\on{pt}]\psi^{d-2},\tau_D,\ldots,\tau_D\rangle_{0,1+l,\beta^\prime}^X y^{\beta^\prime}\right]_{y^\beta}y^\beta (d-1)!\\
&=1+\sum_{\substack{\beta\in \NE(X) \\ d=D\cdot\beta\geq 2}}\left[\sum_{l\geq 0,\beta^\prime\in \NE(X)}\frac{1}{l!}\langle [\on{pt}]\psi^{d-2},\tau_D,\ldots,\tau_D\rangle_{0,1+l,\beta^\prime}^X y^{\beta^\prime}\right]_{y^\beta}y^\beta (d)!.
\end{align*}

We can run the above argument backward to prove Theorem \ref{thm-proper-potential} from Theorem \ref{thm-classical-quantum}. %Note that both theorems are proved using the relative mirror theorem of $(X,D)$. 
Therefore, quantum periods together with the mirror maps give exactly the same information as proper potentials.

\bibliographystyle{amsxport}
\bibliography{main}

\end{document}